\documentclass[10pt]{article}
\usepackage{authblk}
\usepackage{xcolor}
\definecolor{azure}{rgb}{0.25, 0.41, 0.88}
\usepackage{times}
\usepackage{mathtools}
\usepackage[
  colorlinks   = true,   
  linkcolor    = azure,  
  citecolor    = azure, 
  urlcolor     = azure,  
  linktoc      = page    
]{hyperref}
\usepackage[shortlabels]{enumitem}
\usepackage{mathrsfs}
\usepackage{mhequ}
\usepackage{color,soul}

\RequirePackage{color}
\RequirePackage[a4paper,top=1.6in,bottom=1.6in,left=1.6in,right=1.6in]{geometry}
\RequirePackage{lastpage}
\usepackage[english]{babel}
\RequirePackage{amsmath,amsthm,amssymb,mathtools}
\RequirePackage{enumitem}
\RequirePackage{placeins}
\usepackage{tikz}
\usetikzlibrary{arrows,decorations.pathmorphing,backgrounds,positioning,fit,petri}
\usetikzlibrary{shapes}
\usetikzlibrary{shapes.misc}
\tikzset{cross/.style={cross out, draw=black, minimum size=2*(#1-\pgflinewidth), inner sep=0pt, outer sep=0pt},
cross/.default={3pt}}
\usepackage{fancyhdr}
\usepackage[font=small]{caption}
\usepackage{titlesec}
\titleformat{\section}
  {\normalfont\fontsize{14}{15}\bfseries}{\thesection}{1em}{}
\titleformat{\subsection}
  {\normalfont\fontsize{12}{15}\bfseries}{\thesubsection}{1em}{}

\newcommand{\expected}{\mathbb{E}}
\newcommand{\prob}{\mathbb{P}}
\newcommand{\eps}{\varepsilon}

\DeclareMathOperator*{\mymax}{\text{\normalfont max}}
\DeclareMathOperator*{\mymin}{\text{\normalfont min}}
\DeclareMathOperator*{\mysup}{\text{\normalfont sup}}

\DeclareMathOperator*{\mylim}{\text{\normalfont lim}}
\DeclareMathOperator*{\mylog}{\text{\normalfont log}}
\DeclareMathOperator*{\mylimsup}{\text{\normalfont limsup}}

\DeclareMathOperator*{\de}{\hspace{-.1em}\text{\normalfont d\hspace{-.15em}}}
\DeclareMathOperator{\Var}{Var}

\numberwithin{equation}{section}
\newtheorem{theorem}{Theorem}[section]
\newtheorem{corollary}[theorem]{Corollary}
\newtheorem{lemma}[theorem]{Lemma}

\newtheorem{proposition}[theorem]{Proposition}
\theoremstyle{definition}
\newtheorem{definition}[theorem]{Definition}
\newtheorem{remark}[theorem]{Remark}
\newtheorem*{acknowledgements}{Acknowledgements}

\DeclareFontEncoding{LS1}{}{}
\DeclareFontSubstitution{LS1}{stix}{m}{n}
\DeclareSymbolFont{stixletters}{LS1}{stix}{m}{it}
\DeclareMathAccent{\cev}{\mathord}{stixletters}{"91}
\DeclareMathAccent{\vec}{\mathord}{stixletters}{"92}

\makeatletter
\newcommand\@affiliationlist{}
\renewcommand\affil[2][]{%
  \g@addto@macro\@affiliationlist{%
    \item[\textsuperscript{\scriptsize #1}] #2%
  }%
}

\newlength{\affillabelwidth}
\renewcommand{\maketitle}{
  \vspace*{0em}%
  \begin{center}{\LARGE\bfseries \@title \par}\end{center}
  \vspace{1em}%
  \begin{center}{\small \scshape\@date \par}\end{center}\par 
  \vspace{4em}
  \begin{list}{}{\setlength{\leftmargin}{1.5em}\setlength{\rightmargin}{1.5em}}
    \item[] {\raggedright\large\@author \par}
  \end{list}
  \settowidth{\affillabelwidth}{\textsuperscript{\scriptsize 8}} 
  \begin{list}{}{%
      \setlength{\leftmargin}{2.2em}
      \setlength{\rightmargin}{1.5em}
      \setlength{\labelwidth}{\affillabelwidth}
      \setlength{\labelsep}{0.2em}
      \setlength{\itemsep}{.2em}
      \setlength{\parsep}{0pt}
      \setlength{\topsep}{0pt}
  }
    {\raggedright\small \@affiliationlist}
  \end{list}\vspace{0.5em}
}

\renewenvironment{abstract}
  {\par                               
   \begin{list}{}{                        
      \setlength{\leftmargin}{1.5em}
      \setlength{\rightmargin}{1.5em}}
   \item[]                                   
   {\bfseries\abstractname}\par               
   \small\noindent\ignorespaces}             
  {\end{list}}

\begin{document}

\allowdisplaybreaks

\date{\today}

\title{Stochastic Burgers Equation from Non-Product Stationary Measures via a Generalised Second-Order Boltzmann-Gibbs Principle}

\author{\bfseries P. Gonçalves\textsuperscript{1}, M. C. Ricciuti\textsuperscript{2}, G. M. Schütz\textsuperscript{1,3}
}

\affil[1]{Center for Mathematical Analysis, Geometry and Dynamical Systems, Instituto Superior Técnico, Universidade de Lisboa, Av. Rovisco Pais, 1049-001 Lisboa, Portugal}
\affil[2]{Department of Mathematics, Imperial College London, 180 Queen's Gate, London SW7 2AZ, United Kingdom}
\affil[3]{Institute for Advanced Simulation, Forschungszentrum Jülich, 52428 Jülich, Germany \\\vspace{.7em}Emails: \href{mailto:pgoncalves@tecnico.ulisboa.pt}{\texttt{pgoncalves@tecnico.ulisboa.pt}}, \\\href{mailto:maria.ricciuti18@imperial.ac.uk}{\texttt{maria.ricciuti18@imperial.ac.uk}}, \\\href{mailto:gunter.schuetz@tecnico.ulisboa.pt}{\texttt{gunter.schuetz@tecnico.ulisboa.pt}}}

\clearpage\maketitle
\thispagestyle{empty}

\begin{abstract}
We prove a generalised second-order Boltzmann-Gibbs principle for conservative interacting particle systems on a lattice whose stationary measures are not of product type and not invariant under particle jumps. The result, which requires neither a spectral gap bound nor an equivalence of ensembles, extends the classical framework to settings with correlated invariant measures and is based on quantitative bounds for the correlation decay. As an application, we show that the equilibrium density fluctuations of the Katz-Lebowitz-Spohn model with a given choice of parameters converge, under diffusive scaling, to the stationary energy solution of the stochastic Burgers equation. 
\end{abstract}

%%%%%%%%%%%%%%%%%%%%%%%%%%%%%%%%%%%%%%%%%%%%%%%%%%
%%%CONTENTS%%%%%%%%%%%%%%%%%%%%%%%%%%%%%%%%%%%%%%%
%%%%%%%%%%%%%%%%%%%%%%%%%%%%%%%%%%%%%%%%%%%%%%%%%%
%\setcounter{tocdepth}{1}
\tableofcontents

%%%%%%%%%%%%%%%%%%%%%%%%%%%%%%%%%%%%%%%%%%%%%%%%%%
%%%INTRODUCTION%%%%%%%%%%%%%%%%%%%%%%%%%%%%%%%%%%%
%%%%%%%%%%%%%%%%%%%%%%%%%%%%%%%%%%%%%%%%%%%%%%%%%%
\section{Introduction}
In \cite{br84}, the authors introduced the celebrated \textit{Boltzmann-Gibbs principle}, a fundamental result in the field of statistical physics. Roughly speaking, it states that the space-time fluctuations of any field associated with a conservative model can be projected onto the fluctuation field of the conserved quantity. This result formalises the idea that non-conserved quantities fluctuate in a faster time scale than conserved quantities do, and thus, in the characteristic time scale of the conserved quantity, they average out and only their projection onto the conserved field survives in the limit. In \cite{gj14}, the authors extended that result as a \textit{second-order Boltzmann-Gibbs principle}, which states that the first-order correction of the limit in the classical principle can be written as a quadratic functional of the fluctuation field of the conserved quantity. This refinement allows to show convergence of the equilibrium fluctuations of a large class of weakly-asymmetric diffusive models to the \textit{energy solutions} of the stochastic Burgers or the Kardar-Parisi-Zhang (KPZ) equation, also introduced in \cite{gj14} and whose uniqueness in law was later proved in \cite{gp18}. 

One of the key ingredients of the proof in \cite{gj14} consists in the knowledge of a \textit{spectral gap inequality} and an \textit{equivalence of ensembles} for the dynamics. These assumptions are dropped in the subsequent paper \cite{gjs17}, which provides a proof for fields whose additive functionals can be written as polynomials. However, that work still assumed that the stationary measures are of product form and invariant under permutations of nearest-neighbour coordinates. Under these hypotheses, the stochastic Burgers or KPZ equation could be derived as the scaling limit of the fluctuation field for a large class of particle systems. 

Here, we generalise the result and establish what we call the \textit{generalised second-order Boltzmann-Gibbs principle} for particle systems evolving on a finite lattice with $N$ sites, or on the whole $\mathbb Z$. As in \cite{gjs17}, our bound requires neither a spectral gap inequality nor an equivalence of ensembles. Moreover, the stationary measures need not be of product form and invariant under nearest-neighbour swaps, as long as this invariance holds up to a correction of order at most $N^{-\alpha}$ for some $\alpha>0$, and the $m$-point correlation functions decay at least as fast as $N^{-\beta}$ for some $\beta>0$. 

Our arguments are based on a multi-scale analysis, first introduced in \cite{gon08} and then refined in \cite{gj14, gjs17}. The method proceeds by induction, progressively replacing averages of occupation variables over boxes of increasing size so that, in the final step, the box reaches a microscopic length proportional to the system size. At each step of the procedure, occupation variables redistribute mass across the corresponding box in a uniform way. However, in the present setting, the lack of variable-swap invariance introduces new error terms that must be controlled using bounds on the correlation functions.

As an application, we consider the Katz-Lebowitz-Spohn model \cite{kls84} with a specific choice of parameters. This is an asymmetric exclusion process on the discrete torus with a weak hopping bias of order $N^{-\gamma}$ and with next-nearest-neighbour interaction such that the stationary measure is neither product (namely does not factorise over lattice sites) nor invariant under nearest-neighbour swaps. Using the generalised second-order Boltzmann-Gibbs principle, we show that its equilibrium fluctuations converge to the stationary energy solution of the stochastic Burgers equation. To the best of our knowledge, this constitutes the first derivation of the stochastic Burgers or KPZ equation at stationarity as a scaling limit in such a setting, namely from a discrete model whose stationary measures lack both factorisation and swap invariance. 

In a degenerate limit of our parameter choice, the KLS model reduces to the so-called facilitated exclusion process (FEP) for which convergence of the equilibrium fluctuations to the stochastic Burgers equation was proved in \cite{ez24}. However, the apparent absence of factorisation of the invariant measure of the FEP merely reflects the fact that, at fixed particle number, it is a uniform distribution of dimers \cite{ss04} -- the natural analogue of the uniform distribution of monomers arising from restricting a product measure to configurations with a fixed number of particles. This uniformity property, a direct consequence of factorisation at fixed particle number, allows one to derive the convergence to the stochastic Burgers equation by standard techniques via a mapping to the zero-range process (whose invariant measure is product) \cite{ss04,ez24}. This stands in contrast to the parameter regime studied here, where the stationary measure possesses no analogous uniformity property. Moreover, our result provides the first rigorous derivation of predictions from mode-coupling theory for particle systems with non-product invariant measures. Mode-coupling theory, based on the phenomenological framework of non-linear fluctuating hydrodynamics \cite{bks85}, predicts for weakly driven systems a crossover from the superdiffusive KPZ universality class to the diffusive Edwards–Wilkinson (EW) class, with a critical exponent $\gamma = \tfrac{1}{2}$ at which convergence to the stochastic Burgers equation is conjectured \cite{schu24}, as proved below.

The derivation of the stochastic Burgers equation follows the strategy outlined in \cite{gj14}: we first prove tightness of the sequence of fluctuation fields and then identify all limit points as energy solutions. In our setting, since the invariant measure is not of product form, additional work is required to control the fluctuation field at the initial time. Under the invariant measure, occupation variables are identically distributed but not independent, preventing a direct application of the classical central limit theorem. Nevertheless, our control of correlation decay yields quantitative bounds on the $\alpha$-mixing coefficients, allowing us to apply a generalised central limit theorem for $\alpha$-mixing sequences and thereby establish the desired Gaussian convergence. 

A further technical difficulty stems from the degeneracy of the KLS dynamics considered here, in the sense that the system admits blocked configurations. This phenomenon also appeared in the porous media model studied in \cite{glt09, bgs17}, and the trick therein was to split the configuration space into ``good" and ``bad" sets. The good configurations are those that contain at least a mobile cluster -- a pair of particles at distance at most two -- which allows to construct paths that transport the mass across the whole system. The correlation bounds available here enable us to control contributions on the set of bad configurations. Hence, in order to apply the generalised second-order Boltzmann-Gibbs principle in this degenerate setting, we successfully combine our argument with the techniques developed in \cite{bgs17}. 

We conclude with a few directions for future work. It would be interesting to analyse the extension of our generalised second-order Boltzmann-Gibbs principle to the case when the invariant measures are no longer translation invariant. Another challenging problem concerns the extension of our results to non-gradient systems, where the proof is expected to be technically more involved. For the non-gradient KLS model, we expect analogous macroscopic behaviour. A related non-gradient setting that may allow one to address fluctuations near phase transitions is the locally constrained exclusion process \cite{gsz25}, where the invariant measure undergoes a transition from a metastable clustered phase -- reminiscent of phase separation in systems with long-range interactions \cite{ks17} -- to a poorly understood homogeneous regime approaching the stationary measure of the facilitated exclusion process in a suitable limit.

\begin{acknowledgements} \small{P.G. expresses warm thanks to Fundação para a Ciência e Tecnologia FCT\slash Portugal for financial support through the projects UIDB\slash04459\slash2020, UIDP\slash04459\slash2020 and ERC\slash FCT.  M.C.R. gratefully acknowledges support from the Dean's PhD Scholarship at Imperial College London, the kind hospitality of Instituto Superior Técnico in March-April 2025 (when part of this work was carried out), and G-Research for supporting this visit through a travel grant. G.M.S. acknowledges financial support from FCT through the grant 2020.03953.CEECIND.}
\end{acknowledgements}

%%%%%%%%%%%%%%%%%%%%%%%%%%%%%%%%%%%%%%%%%%%%%%%%%%
%%%STATEMENT OF RESULTS%%%%%%%%%%%%%%%%%%%%%%%%%%%
%%%%%%%%%%%%%%%%%%%%%%%%%%%%%%%%%%%%%%%%%%%%%%%%%%
\section{Statement of Results}

%%%%%%%%%%%%%%%%%%%%%%%%%%%%%%%%%%%%%%%%%%%%%%%%%%
%%%The Generalised Second-Order BG Principle%%%%%%
%%%%%%%%%%%%%%%%%%%%%%%%%%%%%%%%%%%%%%%%%%%%%%%%%%
\subsection{The Generalised Second-Order Boltzmann-Gibbs Principle}\label{subsec:gen_bg}
Our framework is the following. We consider a conservative Markov process $\{\eta_{tN^a}^N, t\ge0\}$ in the accelerated time scale $tN^a$ with generator $N^a\mathcal{L}^N$, where $N$ is a scaling parameter and $a>0$; to ease the notation, we will frequently drop the superscript $N$ and simply denote the process by $\{\eta_{tN^a}, t\ge0\}$. This system evolves the discrete lattice $\Lambda_N$ given by the discrete torus $\mathbb{T}_N$ with $N$ points, which we read as a discretisation of the continuum space $\Lambda$ (given by the one-dimensional torus) scaled by a factor $N$. 
To simplify the exposition, we assume that the process is of exclusion type, so that its state space is given by $\Omega_N:=\{0, 1\}^{\Lambda_N}$; however, our results can readily be adapted to more general processes evolving on $\mathcal{X}^{\Lambda_N}$ where $\mathcal{X}$ is a \textit{bounded} space. For $x, y\in\Lambda_N$, let $\eta^{x, y}$ denote the configuration obtained from $\eta\in\Omega_N$ by swapping the states $\eta(x)$ and $\eta(y)$, namely
\begin{equation}\label{eq:swapped_config}
    \eta^{x, y}(z):=\begin{cases} \eta(y) & z=x, \\ \eta(x) & z=y, \\ \eta(z) & z\ne x, y. \end{cases}
\end{equation}
Observe that, above, we read $\eta^{N,N+1}$ as swapping the variables $\eta(N)$ and $\eta(1)$, as we identify the sites $N+1$ and $1$. Throughout, we will mainly use the notation $\eta_x$ to denote the occupation variable $\eta(x)$.

\subsubsection{Assumptions}
\begin{enumerate}

    \item[] {\bfseries \textsc{Assumption 1} (Stationary Measures).} The model admits a family of stationary measures $\{\nu_{\rho}^N\}_{\rho\in I}$ for some $I\subset\mathbb{R}$, and for each $\rho\in I$ we assume that:
    \begin{enumerate}[i)]
        \item $\nu_\rho^N$ is invariant by translation, so that, denoting by $E_{\nu_\rho^N}[\,\cdot\,]$ the expectation with respect to $\nu_\rho^N$, we have that $E_{\nu_\rho^N}[\eta_x]=:\rho_N$ for each $x\in\Lambda_N$;

        \item the measure $\nu_\rho^N$ is invariant under permutations of nearest-neighbouring coordinates up to a correction of order at most $N^{-\alpha}$ for some $\alpha>0$, and this correction depends on the centred configuration only \textit{locally} and \textit{polynomially} up to an error of order at most $N^{-2\alpha}$: namely, calling $\bar\eta_x$ the centred random variable $\eta_x-\rho_N$, we can write
        \begin{equation}\label{eq:poly_correction}
            \frac{\nu_\rho^N(\eta^{x, x+1})}{\nu_\rho^N(\eta)}=1+N^{-\alpha} \sum_{\Theta\subset\{x-m_1, \ldots, x+m_2\}} c_\Theta \prod_{j\in\Theta}\bar\eta_j+O(N^{-2\alpha}),
        \end{equation}
        where $m_1\ge0, m_2\ge1$ and $\{c_\Theta\}_\Theta$ are constants independent of $N$.
    
    \end{enumerate}
\end{enumerate}

This assumption may seem restrictive, but is actually satisfied by a large class of models; for example, any Gibbs measures of a Hamiltonian with short-range polynomial interactions and with temperature scaling as $N^\alpha$ satisfies \eqref{eq:poly_correction}. As we shall see, we can get a better bound for models for which this polynomial has an additional property:
\begin{enumerate}

    \item[] {\bfseries \textsc{Assumption 1a} (Stationary Measures).} The model satisfies \textbf{\textsc{Assumption 1}}, and the expression in \eqref{eq:poly_correction} is such that, defining the coefficients $\{c_\Theta'\}_{\Theta}$ via
    \begin{equation}\label{eq:poly_correctionA}
        \begin{split}\left[\frac{\nu_\rho^N(\eta^{x, x+1})}{\nu_\rho^N(\eta)}-1\right](\bar\eta_x-\bar\eta_{x+1})&=:N^{-\alpha} \sum_{\Theta\subset\{x-m_1, \ldots, x+m_2\}} c_\Theta' \prod_{j\in\Theta}\bar\eta_j
        \\&\phantom{:=}+O(N^{-2\alpha}),
        \end{split}
    \end{equation}
    these satisfy $\{\Theta: \Theta\ne\varnothing \wedge c_\Theta'\ne0\}\ne\varnothing$. 

\end{enumerate}

The term on the left-hand side of \eqref{eq:poly_correctionA} will naturally appear in the proof of our main results, as a consequence of the fact that the underlying measure is not invariant for the swap of variables $\eta\mapsto \eta^{x,x+1}$, which means that every time we make that change of variables, we have a price to pay which is measured in the last assumption. Therefore, this term will appear repeatedly in the proof of our Boltzmann-Gibbs principle, namely Theorem \ref{thm:gen_bg}. In essence, \textbf{\textsc{Assumption 1a}} asserts not only that this expression decays as $N^{-\alpha}$, while being local and polynomial in the centred occupation variables, but also that its constant term vanishes. Consequently, when multiplying two such expressions, provided the sites involved are sufficiently far apart, no constant term survives and only products of occupation variables remain. Together with suitable decay estimates for correlations, this feature allows us to sharpen certain upper bounds. Our next assumption then concerns the asymptotic behaviour of the $m$-point correlation functions under the stationary measures $\{\nu_\rho^N\}_{\rho\in I}$.
\begin{enumerate}

    \item[] {\bfseries \textsc{Assumption 2} (Correlation Functions).} For each $\rho\in I$ and for each integer $m\ge2$, the $m$-point correlation functions of $\nu_\rho^N$ decay at least as fast as $N^{-\beta}$, namely, for each $x_1, \ldots, x_{m}$ in $\Lambda_N$ pairwise distinct,
    \begin{equation*}
        \left|E_{\nu_\rho^N}\left[\prod_{j=1}^{m}\bar\eta_{x_j}\right]\right| \lesssim N^{-\beta}
    \end{equation*}
    for some $\beta>0$.

\end{enumerate}

We note that, above, the decay of the $m$-point correlation function is given in terms of a parameter $\beta$ which in fact usually depends on the value of $m$, as the higher the value of $m$, the faster is the decay to zero of the correlation function; however, here we do not need to precise such dependence. 

Now let $\mathscr{D}_N(f, \nu)$ denote the \textit{Dirichlet form} of the process $\{\eta_t^N, t\ge0\}$ with respect to {a measure $\nu$ on $\Omega_N$}, defined by
\begin{equation}\label{eq:def_dirichlet}
    \mathscr{D}_N(f, \nu) := -\int_{\Omega_N}f\mathcal{L}^N f\de \nu(\eta),
\end{equation} and, for each $x\in\Lambda_N$, let $\tau_x$ denote the translation operator by $x$, so that $\tau_x\eta_\cdot=\eta_{\cdot+x}$.

\begin{enumerate}

    \item[] {\bfseries \textsc{Assumption 3} (Dirichlet Form).} There exists a function $q_{0, 1}^N:\Omega_N\to[\delta, \frac{1}{\delta}]$ with $\delta>0$ such that, for each $\rho\in I$, the Dirichlet form $\mathscr{D}_N(f, \nu_\rho^N)$ can be written as
    \begin{equation*}
        \mathscr{D}_N(f, \nu_\rho^N) = \sum_{x\in\Lambda_N} 
        \int_{\Omega_N} q_{x, x+1}^N(\eta)\left[f(\eta^{x, x+1})-f(\eta)\right]^2\de \nu_\rho^N(\eta),
    \end{equation*}
    where $q_{x, x+1}^N(\eta):=q_{0, 1}^N(\tau_x\eta)$.

\end{enumerate}

\subsubsection{Main Result}
Throughout, we will denote by $\mathbb{P}_{\nu_\rho^N}$ the probability measure on the space of càdlàg trajectories from $[0, T]$ to $\Omega_N$ corresponding to the process $\{\eta_t^N, t\in[0, T]\}$ started at its stationary measure $\nu_\rho^N$, and by $\expected_{\nu_\rho^N}[\,\cdot\,]$ expectations with respect to $\mathbb{P}_{\nu_\rho^N}$. 

Given a function $G:\Lambda\to\mathbb{R}$, let
\begin{equation*}
    \|G\|^2_{2, N}:=\frac{1}{N}\sum_{x\in\Lambda_N}G\left(\frac{x}{N}\right)^2.
\end{equation*}
Also, given an integer $1\le L\le N$ {and $x\in\Lambda_N$, we denote} the left and right averages over boxes of size $L$ as
\begin{equation*}
    \cev{\eta}^{L}(x):=\frac{1}{L}\sum_{y=x-L}^{x-1}\bar\eta_y\ \ \ \text{and} \ \ \ \vec{\eta}^{L}(x):=\frac{1}{L}\sum_{y=x+1}^{x+L}\bar\eta_y,
\end{equation*}
respectively, where we read elements modulo $N$, so that for instance $\cev{\eta}^L(1)=\frac{1}{L}(\bar\eta_{N+1-L}+\ldots+\bar\eta_{N})$. Throughout, even if $L$ is not an integer, we will keep the notation $L$ in the superscript to denote $\lfloor L\rfloor$, {and we will interchangeably denote $\cev{\eta}^{L}(x)$ and $\vec{\eta}^{L}(x)$ by $\cev{\eta}^{L}_x$ and $\vec{\eta}^{L}_x$, respectively.}

\begin{theorem}[Generalised Second-Order Boltzmann-Gibbs Principle]\label{thm:gen_bg} Let \textbf{\textsc{Assumptions 1, 2}} and \textbf{\textsc{3}} hold. 
\begin{enumerate}[i)]

    \item There exists a constant $C=C(\rho)>0$ such that, for any $1\le L\le N$ and $t>0$ and for any function $G:\Lambda\to\mathbb{R}$ such that $\|G\|_{2, N}^2<\infty$,
    \begin{align}
        &\expected_{\nu_\rho^N}\left[\left(\int_0^t\sum_{x\in \Lambda_N}G\left(\frac{x}{N}\right)\left\{\bar\eta_{sN^a}(x)\bar\eta_{sN^a}(x+1)-\cev{\eta}_{sN^a}^L(x)\vec{\eta}_{sN^a}^L(x)\right\}\de s\right)^2\right]\nonumber
        \\&\le Ct\|G\|_{2, N}^2L\left\{\frac{1}{N^{a-1}}+\frac{1}{N^{a+\beta-2}}+\frac{t}{N^{2\alpha-2}}+\frac{t}{N^{2\alpha+\beta-3}}+\frac{t}{N^{3\alpha-3}}\right\};\label{eq:bg_bound}
    \end{align}

    \item if, additionally, \textbf{\textsc{Assumption 1a}} holds, then the bound can be improved to 
    \begin{equation}\label{eq:bg_boundA}
        Ct\|G\|_{2, N}^2L \bigg\{\frac{1}{N^{a-1}}+\frac{1}{N^{a+\beta-2}}+\frac{t}{N^{2\alpha-1}}+\frac{t}{N^{2\alpha+\beta-3}}+\frac{t}{N^{3\alpha-3}}\bigg\}. 
    \end{equation}

\end{enumerate}
\end{theorem}

\begin{remark}
{Note that the bounds \eqref{eq:bg_bound} and \eqref{eq:bg_boundA} are almost the same except for the middle term, whose denominator improves from $N^{2\alpha-2}$ to $N^{2\alpha-1}$. Moreover, we emphasise that the first term in both bounds is the same as in the original result of \cite{gj14, gjs17}; all the additional terms come from the non-vanishing correlations and the fact that the stationary measures are not (necessarily) invariant for the swap of occupation variables.}
\end{remark}

As a straightforward corollary, we recover an analogue of the result in \cite{gjs17} for diffusive models {($a=2$)}, which allows us to ``reach" a box of size $\eps N$:

\begin{corollary}\label{corol:gen_bg} Let \textbf{\textsc{Assumptions 1, 2}} and \textbf{$\textsc{3}$} hold and assume $a=2$.
\begin{enumerate}[i)]

    \item If $\alpha\ge\frac{3}{2}$ and $\beta\ge 1$, for each $t>0$ and $G:\Lambda\to\mathbb{R}$ such that $\|G\|_{2, N}^2<\infty$, we have that
    \begin{equation}\label{eq:bg_limit}
        \begin{split}&\mylim_{\eps\to0}\mylim_{N\to\infty}\expected_{\nu_\rho^N}\Bigg[\bigg(\int_0^t\sum_{x\in \Lambda_N}G\left(\frac{x}{N}\right)\big\{\bar\eta_{sN^2}(x)\bar\eta_{sN^2}(x+1)
        \\&\phantom{\mylim_{\eps\to0}\mylim_{N\to\infty}\expected_{\nu_\rho^N}\Bigg[}-\cev{\eta}_{sN^2}^{\eps N}(x)\vec{\eta}_{sN^2}^{\eps N}(x)
        \big\}\de s\bigg)^2\Bigg]=0;
        \end{split}
    \end{equation}

    \item if, additionally, \textbf{\textsc{Assumption 1a}} holds, then \eqref{eq:bg_limit} holds true under the weaker conditions $\alpha\ge \frac{4}{3}, \beta\ge1$ and $2\alpha+\beta\ge4$.

\end{enumerate}
\end{corollary}
The thresholds for $\alpha$ and $\beta$ in Corollary \ref{corol:gen_bg} simply come from an optimisation of the bounds \eqref{eq:bg_bound} and \eqref{eq:bg_boundA} with the choice $L=\eps N$. We note that reaching a box of size $L=\eps N$ when $a=2$ is crucial in the argument, as for that specific size of the microscopic box, the term involving a product of averages macroscopically becomes a non-trivial observable of the model.

\begin{remark} While our results are stated for $\Lambda_N$ being the discrete torus with $N$ sites, we note that they can readily be adapted to the cases where $\Lambda_N$ is the segment with $N$ points or the whole lattice $\mathbb{Z}$. The reader can retrace all steps of the proof in Section \ref{sec:proof_gen_bg} and will find that the arguments remain analogous.
\end{remark}

%%%%%%%%%%%%%%%%%%%%%%%%%%%%%%%%%%%%%%%%%%%%%%%%%%
%%%Application to the KLS Model%%%%%%%%%%%%%%%%%%%
%%%%%%%%%%%%%%%%%%%%%%%%%%%%%%%%%%%%%%%%%%%%%%%%%%
\subsection{Application to the KLS Model}
Our second contribution is the application of Theorem \ref{thm:gen_bg} to a specific model which does not satisfy the assumptions of the classical second-order Boltzmann-Gibbs principle \cite{gj14, gjs17}, as its stationary measures are neither product nor invariant for particle jumps. In particular, we show that the equilibrium fluctuations of the Katz-Lebowitz-Spohn (KLS) model, for a given choice of parameters, converge to an appropriately defined solution of the stochastic Burgers equation.

\subsubsection{Dynamics and Invariant Measure}\label{subsec:dynamics_measure}
The \textit{KLS model}, introduced in \cite{kls84}, is a Markov process of exclusion type evolving on  $\mathbb{T}_N$ -- so that its state space is $\Omega_N:=\{0,1\}^{\mathbb T_N}$ -- whose generator acts on functions $f:\Omega_N\to\mathbb{R}$ via
\begin{equation*}
    \mathcal{L}^Nf(\eta)=\sum_{x\in\mathbb{T}_N}c_{x, x+1}^N(\eta)[f(\eta^{x, x+1})-f(\eta)],
\end{equation*}
where $\eta^{x, x+1}$ is defined in \eqref{eq:swapped_config} and where
\begin{equation}\label{eq:rates}
    \begin{split}&c_{x,x+1}^N(\eta):=
    \\&p_N\eta_x(1-\eta_{x+1})\{(1+\kappa_N)(1-\eta_{x-1}+\eta_{x+2})+(1+\eps_N)\eta_{x-1}+(1-\eps_N)\eta_{x+2}\}
    \\&+q_N\eta_{x+1}(1-\eta_x)\{(1+\kappa_N)(1-\eta_{x-1}+\eta_{x+2})+(1-\eps_N)\eta_{x-1}+(1+\eps_N)\eta_{x+2}\}.
    \end{split}
\end{equation}
Here, $p_N, q_N\ge0$ and $\eps_N, \kappa_N\in[-1, 1]$. The dynamics can be described as follows: a particle at $x$ waits, independently of the others, a random exponential time of parameter 1, after which it jumps to a neighbouring site $y=x\pm 1$ with rate $c^N_{x, y}(\eta)$; in particular, if the site is unoccupied, then the jump is suppressed. 

For any choice of the parameters $p_N, q_N, \eps_N$ and $\kappa_N$, the Gibbs measure
\begin{equation}\label{eq:nuN}
    \nu_N(\eta): \propto \left(\frac{1+\eps_N}{1-\eps_N}\right)^{-H(\eta)}
\end{equation}
with Hamiltonian
\begin{equation*}
    H(\eta):=\sum_{y\in\mathbb{T}_N}\eta_y\eta_{y+1}
\end{equation*}
is invariant for the dynamics; see \cite[Section 4.1]{ns25}.

Note that, for $\kappa_N=-1$ and $\eps_N=0$, this corresponds to the \textit{porous media model}, namely the model with degenerate rates and  whose hydrodynamic limit in the symmetric case and evolving on $\mathbb{T}_N$ was studied in \cite{glt09}, and equilibrium  fluctuations for the system evolving on $\mathbb{Z}$ was studied in \cite{bgs17}. For this model, the authors of \cite{bgs17} -- after showing that one can restrict to an appropriate set of configurations -- applied the second-order Boltzmann-Gibbs Principle of \cite{gjs17} and showed that, under diffusive scaling ($a=2$), the equilibrium density fluctuations converge to the stationary energy solution of the stochastic Burgers equation. Indeed, if $\eps_N=0$, then the measure \eqref{eq:nuN} becomes the uniform measure, and indeed one can show that, both on $\mathbb{Z}$ and $\mathbb T_N$, the Bernoulli product measures of any parameter $\rho\in(0, 1)$ are invariant for the dynamics, so that the model falls into the assumptions of \cite{gjs17}. However, when $\eps_N\ne0$, \eqref{eq:nuN} is neither product nor invariant for particle jumps, so that those assumptions fail. On the other hand, when
$\kappa_N=-1$ and $\eps_N=1$, the KLS model reduces to the
facilitated exclusion process and, as previously discussed, the measure is then uniform for any fixed particle number (interpreted as dimers), so that -- up to the technicality of equivalence of ensembles -- this model also effectively falls into the assumptions of \cite{gjs17}.

We consider the KLS model evolving on $\mathbb{T}_N$ where we set 
\begin{equation}\label{eq:params}
    \kappa_N:=-1, \ \ \ p_N:=\frac{1}{2}+\frac{b}{2N^\gamma}, \ \ \ q_N:=\frac{1}{2}-\frac{b}{2N^\gamma}, \ \ \ \eps_N:=\frac{1-e^{-N^{-\alpha}}}{1+e^{-N^{-\alpha}}}
\end{equation}
with $b\in\mathbb{R}$ and $\gamma, \alpha>0$; the dynamics for the choice $\kappa_N=-1$ are represented in Figure \ref{fig:kls}.

%%%%%%%%%%%%%%%%%%%%%%%%%%%%%%%%%%%%%%%%%%%%%%%%%%
\vspace{.8em}
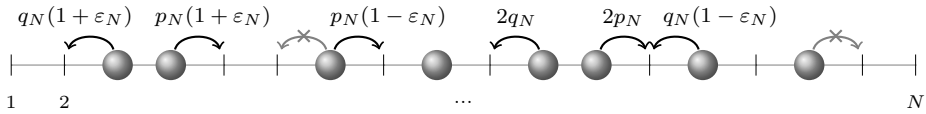
\begin{figure}[!ht]
\begin{center}
\begin{tikzpicture}[thick, scale=0.7][!ht]    
    \draw[step=1cm,gray,very thin] (-9,0) grid (8,0);   
    \foreach \y in {-9,...,8}{
    \draw[color=black,very thin] (\y,-0.2)--(\y,0.2);}    
    \draw (-9,-0.7) node [color=black] {$\scriptstyle{1}$};
    \draw (-8,-0.7) node [color=black] {$\scriptstyle{2}$};
    \draw (-0.5,-0.7) node [color=black] {$\scriptstyle{...}$};
    \draw (8,-0.7) node [color=black] {$\scriptstyle{N}$};

    \node[shape=circle, minimum size=0.4cm] (1) at (-9,0) {};    
    \node[shape=circle, minimum size=0.4cm] (2) at (-8,0) {};
    \node[ball color=black!30!, shape=circle, minimum size=0.4cm] (3) at (-7,0) {};
    \node[ball color=black!30!, shape=circle, minimum size=0.4cm] (4) at (-6,0) {};
    \node[shape=circle, minimum size=0.4cm] (5) at (-5,0) {};
    \node[shape=circle, minimum size=0.4cm] (6) at (-4,0) {};
    \node[ball color=black!30!, shape=circle, minimum size=0.4cm] (7) at (-3,0) {};
    \node[shape=circle, minimum size=0.4cm] (8) at (-2,0) {};
    \node[ball color=black!30!, shape=circle, minimum size=0.4cm] (9) at (-1,0) {};
    \node[shape=circle, minimum size=0.4cm] (10) at (0,0) {};
    \node[ball color=black!30!, shape=circle, minimum size=0.4cm] (11) at (1,0) {};
    \node[ball color=black!30!, shape=circle, minimum size=0.4cm] (12) at (2,0) {};
    \node[shape=circle, minimum size=0.4cm] (13) at (3,0) {};
    \node[ball color=black!30!, shape=circle, minimum size=0.4cm] (14) at (4,0) {};
    \node[shape=circle, minimum size=0.4cm] (15) at (5,0) {};
    \node[ball color=black!30!, shape=circle, minimum size=0.4cm] (16) at (6,0) {};
    \node[shape=circle, minimum size=0.4cm] (17) at (7,0) {};
    \node[shape=circle, minimum size=0.4cm] (18) at (8,0) {};
  
    \path [<-] (2) edge[bend left=75] node[above] {{\footnotesize{$q_N(1+\eps_N)\;\;\;\;\;$}}} (3);
    \path [->] (4) edge[bend left=75] node[above] {{\footnotesize{$\;\;\;\;\;p_N(1+\eps_N)$}}} (5);
    \path [<-] (6) edge[bend left=75, color=black!50!] node[above] {} (7);
    \draw (-3.5,0.55) node[cross] [color=black!50!] {};
    \path [->] (7) edge[bend left=75] node[above] {{\footnotesize{$\;\;\;\;\;\;\;\;\;\;p_N(1-\eps_N)$}}} (8);
    \path [<-] (10) edge[bend left=75] node[above] {{\footnotesize{$2q_N$}}} (11);
    \path [->] (12) edge[bend left=75] node[above] {{\footnotesize{$2p_N$}}} (13);
    \path [<-] (13) edge[bend left=75] node[above] {{\footnotesize{$\;\;\;\;\;\;\;\;\;\;\;\;\;\;q_N(1-\eps_N)$}}} (14);
    \path [->] (16) edge[bend left=75, color=black!50!] node[above] {} (17);
    \draw (6.5,0.55) node[cross] [color=black!50!] {};
    
\end{tikzpicture}
\caption{Microscopic dynamics of the KLS model with $\kappa_N=-1$.}\label{fig:kls}
\end{center}
\end{figure}
%%%%%%%%%%%%%%%%%%%%%%%%%%%%%%%%%%%%%%%%%%%%%%%%%%

Then, the invariant measure reads
\begin{equation}\label{eq:nuN_alpha}
    \nu_\alpha^N(\eta):=Z_{N, \alpha}^{-1} e^{-N^{-\alpha}H(\eta)}
\end{equation}
where $Z_{N, \alpha}$ is the partition function, given by
\begin{equation}\label{eq:partition}
    Z_{N, \alpha}=\int_{\Omega_N}e^{-N^{-\alpha}H(\eta)}.
\end{equation}
A simple proof of the stationarity of $\nu_N^\alpha$ is given in Appendix \ref{sec:appendix}. Throughout, we will denote by $E_{\nu_N^\alpha}[\,\cdot\,]$ expectations with respect to this measure.

\begin{remark} The specific choice of $\eps_N$ is only out of convenience; our method actually works for any choice of $\eps_N$ such that $|\eps_N|\lesssim N^{-\alpha}$.
\end{remark}

Roughly speaking, the dynamics are those of a weakly asymmetric exclusion process on the discrete torus but with jumps from $x$ that can only be performed if at least one site ``close'' to $x$ is occupied, and in particular
\begin{enumerate}[i)]

    \item for right jumps, the rate $p_N$ gets {enhanced} if at least the site to the left of $x$ is occupied, and {reduced} if only that at distance $2$ to the right of $x$ is occupied;

    \item for left jumps, the rate $q_N$ gets {enhanced} if at least the site at distance $2$ to the right of $x$ is occupied, and {reduced} if only that to the left of $x$ is occupied;

\end{enumerate}
if both are unoccupied, the jump is suppressed. In fact, the choice $\kappa_N=-1$ makes the dynamics particularly interesting because in this case, as just highlighted, there are ``blocked" configurations which need to be studied, just as in the porous media model. On the other hand, if $\kappa_N\ne1$, one can check that the symmetric part of the current \textit{cannot} be expressed as a gradient. Thus, the model would lie outside the standard framework and would require techniques different from those presented here.

\subsubsection{Stationary Energy Solutions of the Stochastic Burgers Equation}\label{sec:energy_solutions}
In order to state our convergence result, we first recall the notion of stationary energy solutions of the stochastic Burgers equation, introduced in \cite{gj14}. The definitions and results that follow are taken from \cite{gj14, gjs17, gp18}.

Fix a finite time horizon $[0, T]$ and let {$\mathfrak{D}(\mathbb{T})$} denote the set of real-valued smooth functions on the torus $\mathbb{T}$, and $\mathcal{C}([0, T], \mathfrak{D}'(\mathbb{T}))$ the space of continuous paths taking values in the dual $\mathfrak{D}'(\mathbb{T})$. 

\begin{definition}[Stationarity] We say that a process $\{\mathscr{Y}_t, t\in[0, T]\}$ with trajectories in $\mathcal{C}([0, T], \mathfrak{D}'(\mathbb{T}))$ is \textit{stationary} if, for each $t\in[0, T]$, the $\mathfrak{D}'(\mathbb{T})$-valued random variable $\mathscr{Y}_t$ is a white noise of fixed variance.
\end{definition}

Let $\{\iota_\eps\}_{\eps\in(0, 1)}$ denote an approximation of the identity on $\mathbb{T}$: for each $\phi\in\mathfrak{D}(\mathbb{T})$, we define the process $\{\mathscr{B}_t^\eps, t\in[0, T]\}$ via
\begin{equation*}
    \mathscr{B}_t^\eps(\phi):=\int_0^t\int_{\mathbb{T}}\left(\mathscr{Y}_s\ast \iota_\eps(u)\right)^2\nabla\phi(u)\de u\de s,
\end{equation*}
where $\ast$ denotes the convolution operator. Moreover, for $0\le s\le t\le T$ and $\phi\in\mathfrak{D}(\mathbb{T})$, define $\mathscr{B}_{s, t}^\eps(\phi):=\mathscr{B}_t^\eps(\phi)-\mathscr{B}_s^\eps(\phi)$. We will also denote by $\|\cdot\|_{L^2(\mathbb{T})}$ the norm defined by \begin{equation*}
    \|\phi\|_{L^2(\mathbb{T})}^2:=\int_\mathbb{T}\phi(u)^2\de u
\end{equation*} 
for $\phi\in\mathfrak{D}(\mathbb{T})$.

\begin{definition}[Energy Estimate]\label{def:energy_estimate} We say that a stochastic process $\{\mathscr{Y}_t, t\in[0, T]\}$ with trajectories in $\mathcal{C}([0, T], \mathfrak{D}'(\mathbb{T}))$ satisfies an \textit{energy estimate} if there exists a finite constant $\kappa>0$ such that, for any $0\le s\le t\le T$, any $0<\delta\le\eps<1$ and any $\phi\in\mathfrak{D}(\mathbb{T})$,
\begin{equation*}
    \expected\left[\left(\mathscr{B}_{s, t}^\eps(\phi)-\mathscr{B}_{s, t}^\delta(\phi)\right)^2\right]\le \kappa \eps(t-s)\|\nabla\phi\|_{L^2(\mathbb{T})}^2.
\end{equation*}   
\end{definition}

The following holds.

\begin{proposition} Let $\{\mathscr{Y}_t, t\in[0, T]\}$ be a stochastic process with trajectories in the space $\mathcal{C}([0, T], \mathfrak{D}'(\mathbb{T}))$. Assume that $\{\mathscr{Y}_t, t\in[0, T]\}$ is stationary and satisfies an energy estimate. Then the process $\{\mathscr{B}_t, t\in[0, T]\}$ taking values in $\mathfrak{D}'(\mathbb{T})$ and given by
\begin{equation}\label{eq:mathscr_B}
    \mathscr{B}_t(\phi):=\mylim_{\eps\to0}\mathscr{B}_t^\eps(\phi)
\end{equation}
is well defined and, for any $0\le s\le t\le T$ and any $\phi\in\mathfrak{D}(\mathbb{T})$, it satisfies the estimate
\begin{equation*}
    \expected\left[\left(\mathscr{B}_t(\phi)-\mathscr{B}_s(\phi)\right)^2\right]\le \tilde\kappa \eps(t-s)^{\frac{3}{2}}\|\nabla\phi\|_{L^2(\mathbb{T})}^2
\end{equation*}
for some constant $\tilde\kappa>0$.
\end{proposition}

The proposition above gives meaning to the square of a distribution-valued process $\{\mathscr{Y}_t, t\in[0, T]\}$, provided it satisfies Definition \ref{def:energy_estimate}, as highlighted in the following.

\begin{definition}[Stationary Energy Solution of the Stochastic Burgers Equation]\label{def:energy_solutions} Let $\lambda\in\mathbb{R}, \nu, \sigma>0$, and let $\{\mathscr{W}_t, t\in[0, T]\}$ be a $\mathfrak{D}'(\mathbb{T})$-valued Brownian motion. We say that a stochastic processes $\{\mathscr{Y}_t, t\in[0, T]\}$ with trajectories in $\mathcal{C}([0, T], \mathfrak{D}'(\mathbb{T}))$ is a \textit{stationary energy solution} of the stochastic Burgers equation 
\begin{equation}\label{eq:general_SBE}
    \partial_t\mathscr{Y}_t=\nu\Delta\mathscr{Y}_t+\lambda \nabla(\mathscr{Y}_t^2)+\sqrt{2\nu\sigma^2}\nabla \dot{\mathscr{W}}_t
\end{equation}
if:
\begin{enumerate}[i)]

    \item for each $t\in[0, T]$, the $\mathfrak{D}'(\mathbb{T})$-valued random variable $\mathcal{Y}_t$ is a white noise of variance $\sigma^2$,

    \item the process $\{\mathscr{Y}_t, t\in[0, T]\}$ satisfies an energy estimate,

    \item for any $\phi\in\mathfrak{D}(\mathbb{T})$, the process $\{\mathscr{M}_t(\phi), t\in[0, T]\}$ defined by
    \begin{equation*}
        \mathscr{M}_t(\phi)=\mathscr{Y}_t(\phi)-\mathscr{Y}_0(\phi)-\int_0^t\mathscr{Y}_s(\nu\Delta\phi)\de s+\lambda\mathscr{B}_t(\phi)
    \end{equation*}
    is a one-dimensional Brownian motion of variance $2\nu\sigma^2\|\nabla\phi\|_{|L^2(\mathbb{T})}^2$, with $\{\mathscr{B}_t, t\in[0, T]\}$ defined as in \eqref{eq:mathscr_B},

    \item the reversed processes $\{\mathscr{Y}_{T-t},\mathscr{B}_{T-t}-\mathscr{B}_{T});\, t\in[0, T]\}$ also satisfy item iii), but with $\lambda$ replaced by $-\lambda$.

\end{enumerate}
\end{definition}

It is a known fact that the notion of an energy solution from Definition \ref{def:energy_solutions}, which is slightly different from the one originated in \cite{gj14} due to item iv), gives in fact uniqueness in law:

\begin{proposition} If $\{\mathscr{Y}_t, t\in[0, T]\}$ and $\{\mathscr{Y}'_t, t\in[0, T]\}$ are both stationary energy solutions of \eqref{eq:general_SBE}, they have the same law.
\end{proposition}

The proof of the proposition above can be found in \cite{gp18}.

\subsubsection{Convergence of Equilibrium Fluctuations}
Our process of interest is the diffusively-scaled \textit{density fluctuation field} $\{\mathscr{Y}_t^N, t\ge0\}$, taking values in $\mathfrak{D}'(\mathbb{T})$ and defined, for each $\phi\in\mathfrak{D}(\mathbb{T})$, by
\begin{equation*}
    \mathscr{Y}_t^N(\phi):=\frac{1}{\sqrt{N}}\sum_{x\in\mathbb{T}_N} \phi\left(\frac{x-\boldsymbol{v}_Nt}{N}\right)\bar\eta_{tN^2}^N(x).
\end{equation*}
Here, $\bar\eta_{tN^2}^N(x)$ is the centred random variable $\eta_{tN^2}^N(x)-\bar\rho_N$, where $\bar\rho_N:=E_{\nu_N^\alpha}[\eta_x^N]$, and $\boldsymbol{v}_N$ is the \textit{transport velocity} of the fluctuations. In particular, let
\begin{equation}\label{eq:transport_velocity}
    \boldsymbol{v}_N:=2bN^{2-\gamma}\bar\rho_N\left(2-3\bar\rho_N\right).
\end{equation}
This choice of velocity will become clear in Section \ref{sec:dynkin}, but we anticipate that its role is to remove the degree-one terms of the microscopic current, so that the relevant ones are the quadratic terms. Fix a finite time horizon $[0, T]$, and let $\mathcal{D}([0, T], \mathfrak{D}'(\mathbb{T}))$ denote the space of càdlàg paths taking values in $\mathfrak{D}'(\mathbb{T})$.

\begin{theorem}\label{thm:fluctuations} Assume that $\gamma\ge\frac{1}{2}$ and $\alpha\ge\frac{4}{3}$. The sequence of processes $\{\mathscr{Y}_t^N, t\in[0, T]\}_N$ converges in distribution with respect to the Skorohod topology of $\mathcal{D}([0, T], \mathfrak{D}'(\mathbb{T}))$, and:
\begin{enumerate}[i)]

    \item if $\gamma>\frac{1}{2}$, its limit is the stationary solution of the infinite-dimensional Ornstein-Uhlenbeck equation
    \begin{equation}\label{eq:limit_OU}
        \partial_t\mathscr{Y}_t=\frac{1}{2}\Delta\mathscr{Y}_t+\frac{1}{2}\nabla \dot{\mathscr{W}}_t;
    \end{equation}

    \item if $\gamma=\frac{1}{2}$, its limit is the stationary energy solution of the stochastic Burgers equation
    \begin{equation}\label{eq:limit_SBE}
        \partial_t\mathscr{Y}_t=\frac{1}{2}\Delta\mathscr{Y}_t-b\nabla(\mathscr{Y}_t^2)+\frac{1}{2}\nabla\dot{\mathscr{W}}_t,
    \end{equation}

\end{enumerate}
where $\{\mathscr{W}_t, t\in[0, T]\}$ is an $\mathfrak{D}'(\mathbb{T})$-valued Brownian motion.
\end{theorem}
The restriction $\alpha\ge\frac{4}{3}$ is due to the fact that, from Lemma \ref{lemma:correlations}, \textbf{\textsc{Assumption 2}} is verified for $\beta\ge\alpha$; plugging this in \eqref{eq:bg_boundA} with $L=\eps N$ then yields the lower bound for $\alpha$.

Our proof strategy follows the standard program: we write the martingale equation for the fluctuation field and establish tightness in $\mathcal{D}([0,T], \mathfrak{D}'(\mathbb{T}))$ for each of its terms. The main difficulty, as expected, lies in handling the degree-2 term, for which we invoke Theorem \ref{thm:gen_bg}. However, as emphasised in Section \ref{subsec:dynamics_measure}, \textbf{\textsc{Assumption 3}} is not satisfied for this model, since the rates may be degenerate. Most of the work in Section \ref{sec:kls} is therefore devoted to showing that one can restrict to a suitable set of configurations on which the assumption does hold.

%%%%%%%%%%%%%%%%%%%%%%%%%%%%%%%%%%%%%%%%%%%%%%%%%%
%%%PROOF OF THE GENERALISED SECOND-ORDER BG%%%%%%%
%%%%%%%%%%%%%%%%%%%%%%%%%%%%%%%%%%%%%%%%%%%%%%%%%%
\section{Proof of the Generalised Second-Order Boltzmann-Gibbs Principle}\label{sec:proof_gen_bg}
In this section we work under \textbf{\textsc{Assumptions 1, 2}} and \textbf{\textsc{3}} and we revisit the steps used in \cite{gjs17} to prove the second-order Boltzmann-Gibbs Principle, based on the one-block-two-block scheme introduced in \cite{gpv88}. Compared to the strategy therein, our method consists in adding a correction -- right before using the Kipnis-Varadhan inequality -- that takes into account the non-invariance of the stationary measure for particle jumps. The error term coming from this correction will then yield an additional term in all of our bounds depending on the parameters $\alpha$ and $\beta$ defined in \textbf{\textsc{Assumptions 1}} and \textbf{\textsc{2}}.

%%%%%%%%%%%%%%%%%%%%%%%%%%%%%%%%%%%%%%%%%%%%%%%%%%
%%%One-Block Estimate%%%%%%%%%%%%%%%%%%%%%%%%%%%%%
%%%%%%%%%%%%%%%%%%%%%%%%%%%%%%%%%%%%%%%%%%%%%%%%%%
\subsection{One-Block Estimate}
\begin{lemma}[One-Block Estimate]\label{lemma:one_block} Let $1\le \ell_0\le N$ and let $\varphi:\Omega_N\to\mathbb{R}$ be a function of the form $\varphi(\eta)=\cev{\eta}^\ell_0$ for some $1\le \ell\le N$.
\begin{enumerate}[i)]

    \item There exists a constant $C$ independent of $N$ such that, for any $t>0$ and any $G:\Lambda\to\mathbb{R}$ such that $\|G\|_{2, N}^2<\infty$,
    \begin{align*}
        &\expected_{\nu_\rho^N}\left[\left(\int_0^t\sum_{x\in\Lambda_N} G\left(\frac{x}{N}\right)\varphi(\tau_x\eta_{sN^a})\left[\bar\eta_{sN^a}(x+1)-\vec{\eta}^{\ell_0}_{sN^a}(x)\right]\de s\right)^2\right]
        \\&\le Ct\ell_0^2\|G\|_{2, N}^2\left\{\frac{1}{N^{a-1}\ell}+\frac{1}{N^{a+\beta-1}}+\frac{t(\ell+\ell_0)}{N^{2\alpha-1}\ell}+\frac{t}{N^{2\alpha+\beta-2}}+\frac{t}{N^{3\alpha-2}}\right\};
    \end{align*}

    \item if, additionally, \textbf{\textsc{Assumption 1a}} holds, then the bound can be improved to
    \begin{equation*}
        Ct\ell_0^2\|G\|_{2, N}^2\left\{\frac{1}{N^{a-1}\ell}+\frac{1}{N^{a+\beta-1}}+\frac{t(\ell+\ell_0)}{N^{2\alpha-1}\ell_0\ell}+\frac{t}{N^{2\alpha+\beta-2}}+\frac{t}{N^{3\alpha-2}}\right\}.
    \end{equation*}

\end{enumerate}
The same result holds by substituting $\left[\bar\eta_{x+1}-\vec{\eta}^{\ell_0}_x\right]$ with $\left[\bar\eta_{x}-\cev{\eta}^{\ell_0}_x\right]$, as long as $\varphi$ is of the form $\varphi(\eta)=\vec{\eta}^\ell_0$ for some $1\le \ell\le N$.
\end{lemma}

\begin{remark}\label{remark:varphi} Before proving this lemma, let us emphasise that, unlike the analogous result \cite[Proposition 5]{gjs17}, here (and later on) we impose a specific structure on the function $\varphi$, rather than merely assuming that its support does not intersect the set $\{0, \ldots, \ell_0\}$. This additional requirement stems from the fact that our stationary measures are not of product type, so some information on $\varphi$ is needed in order to estimate the ``crossed'' terms involving it. In fact, $\varphi$ need not have the exact form we impose: any function satisfying the support condition together with an explicit polynomial representation in terms of local, centred occupation variables would suffice. Nevertheless, our chosen structure as a local average is convenient and sufficient to establish Theorem \ref{thm:gen_bg}.
\end{remark}

\begin{proof}[Proof of Lemma \ref{lemma:one_block}]
Throughout, $C$ will denote a generic positive constant independent of $N$ but changing from time to time or even from line to line. For each $z\in\Lambda_N$, let
\begin{equation}\label{eq:sigma}
    \sigma_z=\sigma_z(\eta):=\frac{\nu_\rho^N(\eta^{z, z+1})}{\nu_\rho^N(\eta)}.
\end{equation}
The key observation is that we can write
\begin{align}
    \bar\eta_{x+1}-\vec{\eta}^{\ell_0}_x&=\frac{1}{\ell_0}{\sum_{y=x+2}^{x+\ell_0}}\sum_{z=x+1}^{y-1}\left[\bar\eta_z-\bar\eta_{z+1}\right]\label{eq:factor_one}
    \\&=\vec{\Sigma}_x(\eta, \ell_0)+ \vec{\Lambda}_x(\eta, \ell_0),\label{eq:factor_sigmas}
\end{align}
where
\begin{equation*}
    \vec{\Sigma}_x(\eta, \ell_0):=\frac{1}{\ell_0}\sum_{y=x+2}^{x+\ell_0}\sum_{z=x+1}^{y-1}\left[\bar\eta_z-\bar\eta_{z+1}\right]\left[\frac{1+\sigma_z}{2}\right]
\end{equation*}
and
\begin{equation*}
    \vec{\Lambda}_x(\eta, \ell_0):=\frac{1}{\ell_0}\sum_{y=x+2}^{x+\ell_0}\sum_{z=x+1}^{y-1}\left[\bar\eta_z-\bar\eta_{z+1}\right]\left[\frac{1-\sigma_z}{2}\right].
\end{equation*}
By the inequality $(x+y)^2\le 2x^2+2y^2$, the expectation in the statement is bounded from above by
\begin{align}
    &2\expected_{\nu_\rho^N}\left[\left(\int_0^t\sum_{x\in\Lambda_N} G\left(\frac{x}{N}\right)\varphi(\tau_x\eta_{sN^a})\vec{\Sigma}_x(\eta_{sN^a}, \ell_0)\de s\right)^2\right]\label{eq:Sigma_1block}
    \\&+2\expected_{\nu_\rho^N}\left[\left(\int_0^t\sum_{x\in\Lambda_N} G\left(\frac{x}{N}\right)\varphi(\tau_x\eta_{sN^a})\vec{\Lambda}_x(\eta_{sN^a}, \ell_0)\de s\right)^2\right].\label{eq:Lambda_1block}
\end{align}

We start with \eqref{eq:Sigma_1block}. By the Kipnis-Varadhan inequality \cite[Lemma 2.4]{klo12}, this term satisfies
\begin{align}
    \eqref{eq:Sigma_1block}&\le Ct\left\|\sum_{x\in\Lambda_N} G\left(\frac{x}{N}\right)\varphi(\tau_x\eta)\vec{\Sigma}_x(\eta, \ell_0)\right\|_{-1}^2\nonumber
    \\\begin{split}&=Ct\mysup_{f\in L^2(\nu_\rho^N )}\bigg\{2\int_{\Omega_N}\sum_{x\in\Lambda_N} G\left(\frac{x}{N}\right)\varphi(\tau_x\eta)\vec{\Sigma}_x(\eta, \ell_0)f(\eta)\de \nu_\rho^N-N^a\mathscr{D}_N\left(f, \nu_\rho^N \right)\bigg\}.\label{eq:sup_1block}
    \end{split}
\end{align}
Now note that, by the equality $\eqref{eq:factor_one}=\eqref{eq:factor_sigmas}$, we can write $\vec{\Sigma}_x(\eta, \ell_0)$ as
\begin{equation}\label{eq:S_1block}
    \vec{\Sigma}_x(\eta, \ell_0)= \frac{1}{\ell_0}\sum_{y=x+2}^{x+\ell_0}\sum_{z=x+1}^{y-1}\left[\bar\eta_z-\bar\eta_{z+1}\right]-\vec{\Lambda}_x(\eta, \ell_0).
\end{equation}
Then, the integral term in \eqref{eq:sup_1block} corresponding to the sum on the right-hand side of \eqref{eq:S_1block} can be written as
\begin{align}
    &2\int_{\Omega_N}\sum_{x\in\Lambda_N} G\left(\frac{x}{N}\right)\varphi(\tau_x\eta)\left(\frac{1}{\ell_0}\sum_{y=x+2}^{x+\ell_0}\sum_{z=x+1}^{y-1}\left[\bar\eta_z-\bar\eta_{z+1}\right]\right)f(\eta)\de \nu_\rho^N \nonumber
    \\&=\int_{\Omega_N}\sum_{x\in\Lambda_N} G\left(\frac{x}{N}\right)\varphi(\tau_x\eta)\left(\frac{1}{\ell_0}\sum_{y=x+2}^{x+\ell_0}\sum_{z=x+1}^{y-1}\left[\bar\eta_z-\bar\eta_{z+1}\right]\right)\left[f(\eta)-f(\eta^{z, z+1})\right]\de \nu_\rho^N  \label{eq:firstAdd_1block}
    \\&\phantom{=}+\int_{\Omega_N}\sum_{x\in\Lambda_N} G\left(\frac{x}{N}\right)\varphi(\tau_x\eta)\left(\frac{1}{\ell_0}\sum_{y=x+2}^{x+\ell_0}\sum_{z=x+1}^{y-1}\left[\bar\eta_z-\bar\eta_{z+1}\right]\right)\left[f(\eta)+f(\eta^{z, z+1})\right]\de \nu_\rho^N .\label{eq:secondAdd_1block}
\end{align}
By Young's inequality, for any set of positive real numbers $\{A_x\}_{x\in\Lambda_N}$ we have that \eqref{eq:firstAdd_1block} is bounded from above by
\begin{align*}
    &\frac{1}{\ell_0}\sum_{x\in\Lambda_N}\sum_{y=x+2}^{x+\ell_0}\sum_{z=x+1}^{y-1} G\left(\frac{x}{N}\right)\frac{A_x}{2}\int_{\Omega_N}q^N_{z, z+1}(\eta)\left[f(\eta)-f(\eta^{z, z+1})\right]^2\de \nu_\rho^N  
    \\&+\frac{1}{\ell_0}\sum_{x\in\Lambda_N}\sum_{y=x+2}^{x+\ell_0}\sum_{z=x+1}^{y-1} G\left(\frac{x}{N}\right)\frac{1}{2A_x}\int_{\Omega_N}\varphi(\tau_x\eta)^2\frac{\left[\bar\eta_z-\bar\eta_{z+1}\right]^2}{q_{z, z+1}^N(\eta)}\de \nu_\rho^N,
\end{align*}
where $q_{z, z+1}^N$ is defined in \textbf{\textsc{Assumption 3}}. If we choose $A_x=N^a\left(4\ell_0G(\frac{x}{N})\right)^{-1}$, the first integral term is bounded from above by $N^a\mathscr{D}_N(f, \nu_\rho^N )$. As for the second term, we get
\begin{align}
    &\frac{2}{N^a}\sum_{x\in\Lambda_N}\sum_{y=x+2}^{x+\ell_0}\sum_{z=x+1}^{y-1} G\left(\frac{x}{N}\right)^2\int_{\Omega_N}\varphi(\tau_x\eta)^2\frac{\left[\bar\eta_z-\bar\eta_{z+1}\right]^2}{q_{z, z+1}^N(\eta)}\de \nu_\rho^N \nonumber
    \\&\le \frac{C}{N^a}\sum_{x\in\Lambda_N}\sum_{y=x+2}^{x+\ell_0}\sum_{z=x+1}^{y-1} G\left(\frac{x}{N}\right)^2 E_{\nu_\rho^N}\left[\varphi(\tau_x\eta)^2\right]\nonumber
    \\&\le \frac{C\ell_0^2}{N^{a-1}}\|G\|_{2, N}^2E_{\nu_\rho^N}\left[\varphi(\tau_x\eta)^2\right].\label{eq:pre_bound}
\end{align}
We now turn to \eqref{eq:secondAdd_1block}: by performing the change of variable $\eta\mapsto \eta^{z, z+1}$ and by using the fact that the support of $\varphi$ does not intersect the set of points $\{0, \ldots, \ell_0\}$, we get that this term is equal to
\begin{equation*}
    \int_{\Omega_N}\sum_{x\in\Lambda_N} G\left(\frac{x}{N}\right)\varphi(\tau_x\eta)\left\{\frac{1}{\ell_0}\sum_{y=x+2}^{x+\ell_0}\sum_{z=x+1}^{y-1}\left[\bar\eta_z-\bar\eta_{z+1}\right]\left[1-\frac{\nu_\rho^N (\eta^{z, z+1})}{\nu_\rho^N (\eta)}\right]\right\}f(\eta)\de \nu_\rho^N.
\end{equation*}
But then, this term cancels with the contribution in \eqref{eq:sup_1block} given by the second term on the right-hand side of \eqref{eq:S_1block}, so we get $\eqref{eq:Sigma_1block}\le\eqref{eq:pre_bound}$. Now, using \textbf{\textsc{Assumption 2}}, we see that
\begin{equation}\label{eq:variance_varphi}
    E_{\nu_\rho^N}\left[\varphi(\tau_x\eta)^2\right]=\frac{1}{\ell^2}E_{\nu_\rho^N}\left[\left(\bar\eta_{x-\ell}+\ldots+\bar\eta_{x-1}\right)^2\right]\lesssim\frac{1}{\ell}+\frac{1}{N^\beta}
\end{equation}

for each $x\in\Lambda_N$, which yields
\begin{equation*}
    \eqref{eq:Sigma_1block}\le \frac{C\ell_0^2}{N^{a-1}}\|G\|_{2, N}^2\left(\frac{1}{\ell}+\frac{1}{N^\beta}\right).
\end{equation*}

We are now left to control \eqref{eq:Lambda_1block}. To simplify the exposition, we assume that $m_1, m_2$ defined in \textbf{\textsc{Assumption 1}} are given by $m_1=0, m_2=1$, but the proof can readily be adapted to any finite values. By the Cauchy-Schwarz inequality and the inequality $(x+y)^2\le 2x^2+2y^2$, we have that
\begin{align}
    \eqref{eq:Lambda_1block}&\le Ct^2 E_{\nu_\rho^N}\left[\left(\sum_{x\in\Lambda_N} G\left(\frac{x}{N}\right)\varphi(\tau_x\eta)\vec{\Lambda}_x(\eta, \ell_0)\right)^2\right]\nonumber
    \\&\le Ct^2(\ell+\ell_0)\sum_{x\in\Lambda_N} G\left(\frac{x}{N}\right)^2E_{\nu_\rho^N}\left[\varphi(\tau_x\eta)^2\vec{\Lambda}_x(\eta, \ell_0)^2\right]\label{eq:sumX_1block}
    \\&\phantom{\le}+Ct^2\sum_{|x-y|>\ell+\ell_0} G\left(\frac{x}{N}\right)G\left(\frac{y}{N}\right)E_{\nu_\rho^N}\left[\varphi(\tau_x\eta)\varphi(\tau_y\eta)\vec{\Lambda}_x(\eta, \ell_0)\vec{\Lambda}_y(\eta, \ell_0)\right]. \label{eq:sumXY_1block}
\end{align}
Now:
\begin{enumerate}[i)]
    
    \item by \textbf{\textsc{Assumption 1}}, we see that $\vec{\Lambda}_x(\eta, \ell_0)^2\lesssim \frac{\ell_0^2}{N^{2\alpha}}$, and by \eqref{eq:variance_varphi} we have that $E_{\nu_\rho^N}\left[\varphi(\tau_x\eta)^2\right]\lesssim(\frac{1}{\ell}+\frac{1}{N^\beta})$. Hence, 
    \begin{equation*}
        \eqref{eq:sumX_1block}\le \frac{Ct^2\ell_0^2}{N^{2\alpha-1}}\|G\|_{2, N}^2(\ell+\ell_0)\left(\frac{1}{\ell}+\frac{1}{N^\beta}\right).
    \end{equation*}
    As for \eqref{eq:sumXY_1block}, note that, by \textbf{\textsc{Assumptions 1}} and \textbf{\textsc{2}}, for each $|x-y|>\ell+\ell_0$ we have that
    \begin{align*}
        &\left|E_{\nu_\rho^N}\left[\varphi(\tau_x\eta)\varphi(\tau_y\eta)\vec{\Lambda}_x(\eta, \ell_0)\vec{\Lambda}_y(\eta, \ell_0)\right]\right|
        \\&=\frac{1}{4\ell^2\ell_0^2}\Bigg|E_{\nu_\rho^N}\Bigg[\sum_{i=x-\ell}^{x-1}\bar\eta_i \sum_{j=y-\ell}^{y-1}\bar\eta_j\sum_{x'=x+2}^{x+\ell_0}\sum_{w=x+1}^{x'-1} \overbrace{[\bar\eta_w-\bar\eta_{w+1}][1-\sigma_w]}^{N^{-\alpha}\cdot\text{poly}(\bar\eta_w, \bar\eta_{w+1})+O(N^{-2\alpha})} \times
        \\&\phantom{=\frac{1}{4\ell^2\ell_0^2}\Bigg|E_{\nu_\rho^N}\Bigg[}\times \sum_{y'=y+2}^{y+\ell_0}\sum_{z=y+1}^{y'-1} \underbrace{[\bar\eta_z-\bar\eta_{z+1}][1-\sigma_z]}_{N^{-\alpha}\cdot\text{poly}(\bar\eta_z, \bar\eta_{z+1})+O(N^{-2\alpha})}\Bigg]\Bigg|
        \\&\lesssim \frac{\ell_0^2}{N^{2\alpha+\beta}}+\frac{\ell_0^2}{N^{3\alpha}},
    \end{align*}
    where the notation $\text{poly}(\bar\eta_w, \bar\eta_{w+1})$ is intended in the sense of \eqref{eq:poly_correction}. Thus,
    \begin{equation*}
        \eqref{eq:sumXY_1block}\le \frac{Ct^2\ell_0^2}{N^{2\alpha-2}}\|G\|_{2, N}^2\left(\frac{1}{N^\beta}+\frac{1}{N^\alpha}\right);
    \end{equation*}
    
    \item if, additionally, \textbf{\textsc{Assumption 1a}} holds, then the bound of \eqref{eq:sumX_1block} can be improved: indeed, in this case, for each $x$,
    \begin{align*}
        &E_{\nu_\rho^N}\left[\varphi(\tau_x\eta)^2\vec{\Lambda}_x(\eta, \ell_0)^2\right]
        \\&=\frac{1}{4\ell^2\ell_0^2}\Bigg|E_{\nu_\rho^N}\Bigg[\sum_{i=x-\ell}^{x-1}\bar\eta_i \sum_{j=x-\ell}^{x-1}\bar\eta_j\sum_{x'=x+2}^{x+\ell_0}\sum_{w=x+1}^{x'-1} \overbrace{[\bar\eta_w-\bar\eta_{w+1}][1-\sigma_w]}^{N^{-\alpha}\cdot\text{poly}_{\text{deg}\ge1}(\bar\eta_w, \bar\eta_{w+1})+O(N^{-2\alpha})} \times
        \\&\phantom{=\frac{1}{4\ell^2\ell_0^2}\Bigg|\expected\Bigg[}\times \sum_{y'=y+2}^{x+\ell_0}\sum_{z=x+1}^{y'-1} \underbrace{[\bar\eta_z-\bar\eta_{z+1}][1-\sigma_z]}_{N^{-\alpha}\cdot\text{poly}_{\text{deg}\ge1}(\bar\eta_z, \bar\eta_{z+1})+O(N^{-2\alpha})}\Bigg]\Bigg|
        \\&\lesssim \frac{\ell_0}{\ell N^{2\alpha}}+\frac{\ell_0^2}{\ell N^{2\alpha+\beta}}+\frac{\ell_0^2}{\ell N^{3\alpha}}+\frac{\ell_0}{N^{2\alpha+\beta}}+\frac{\ell_0}{N^{3\alpha}}+\frac{\ell_0^2}{N^{2\alpha+\beta}}+\frac{\ell_0^2}{N^{3\alpha}},
    \end{align*}
    where the notation $\text{poly}_{\text{deg}\ge1}(\bar\eta_w, \bar\eta_{w+1})$ is intended in the sense of \eqref{eq:poly_correctionA}, so that
    \begin{equation*}
        \eqref{eq:sumX_1block}\le \frac{Ct^2\ell_0^2}{N^{2\alpha-1}}\|G\|_{2, N}^2(\ell+\ell_0)\left(\frac{1}{\ell_0\ell}+\frac{1}{N^\beta}+\frac{1}{N^\alpha}\right).
    \end{equation*}
    
\end{enumerate}
Combining everything together, we finally get the required bounds. 
\end{proof}

%%%%%%%%%%%%%%%%%%%%%%%%%%%%%%%%%%%%%%%%%%%%%%%%%%
%%%Doubling the Box%%%%%%%%%%%%%%%%%%%%%%%%%%%%%%%
%%%%%%%%%%%%%%%%%%%%%%%%%%%%%%%%%%%%%%%%%%%%%%%%%%
\subsection{Doubling the Box}
\begin{lemma}[Two-Block Estimate]\label{lemma:two_block} For $k\in\mathbb{N}$, fix $\ell_k\in\mathbb{N}$ and define $\ell_{k+1}:=2\ell_k$. Let $\varphi:\Omega_N\to\mathbb{R}$ be a function of the form $\varphi(\eta)=\cev{\eta}^\ell_0$ for some $1\le \ell\le N$. 

\begin{enumerate}[i)]

    \item There exists a constant $C$ independent of $N$ such that, for any $t>0$ and any $G:\Lambda\to\mathbb{R}$ such that $\|G\|_{2, N}^2<\infty$,
    \begin{align*}
        &\expected_{\nu_\rho^N}\left[\left(\int_0^t\sum_{x\in\Lambda_N} G\left(\frac{x}{N}\right)\varphi(\tau_x\eta_s)\left[\vec\eta^{\ell_k}_{sN^a}(x+1)-\vec{\eta}^{\ell_{k+1}}_{sN^a}(x)\right]\de s\right)^2\right]
        \\&\le  Ct\ell_k^2\|G\|_{2, N}^2\left\{\frac{1}{N^{a-1}\ell}+\frac{1}{N^{a+\beta-1}}+\frac{t(\ell+\ell_{k+1})}{N^{2\alpha-1}\ell}+\frac{t}{N^{2\alpha+\beta-2}}+\frac{t}{N^{3\alpha-2}}\right\};
    \end{align*}
    
    \item if, additionally, \textbf{\textsc{Assumption 1a}} holds, then the bound can be improved to
    \begin{equation*}
        Ct\ell_k^2\|G\|_{2, N}^2\left\{\frac{1}{N^{a-1}\ell}+\frac{1}{N^{a+\beta-1}}+\frac{t(\ell+\ell_{k+1})}{N^{2\alpha-1}\ell_k\ell}+\frac{t}{N^{2\alpha+\beta-2}}+\frac{t}{N^{3\alpha-2}}\right\}.
    \end{equation*}

\end{enumerate}
The same result holds by substituting $\left[\vec\eta^{\ell_k}_{x+1}-\vec{\eta}^{\ell_{k+1}}_x\right]$ with $\left[\cev\eta^{\ell_k}_{x+1}-\cev{\eta}^{\ell_{k+1}}_x\right]$, as long as $\varphi$ is of the form $\varphi(\eta)=\vec{\eta}^\ell_0$ for some $1\le \ell\le N$.
\end{lemma}

\begin{proof}
We proceed in a similar way to the proof of Lemma \ref{lemma:one_block}: we write
\begin{align*}
    \vec{\eta}^{\ell_k}_x-\vec{\eta}^{\ell_{k+1}}_{x+1}&=\frac{1}{2\ell_k}\sum_{y=1}^{\ell_k}\sum_{z=y+x}^{y+x+\ell_k-1}\left[\bar\eta_z-\bar\eta_{z+1}\right]
    \\&=\sum_{y=1}^{\ell_k}\vec{\Sigma}_{x+y}(\eta, \ell_k)+\sum_{y=1}^{\ell_k}\vec{\Lambda}_{x+y}(\eta, \ell_k),
\end{align*}
where
\begin{equation*}
    \vec{\Sigma}_{x}(\eta, \ell_k)=\frac{1}{2\ell_k}\sum_{z=x}^{x+\ell_k-1}\left[\bar\eta_z-\bar\eta_{z+1}\right]\left[\frac{1+\sigma_z}{2}\right],
\end{equation*}
\begin{equation*}
    \vec{\Lambda}_x(\eta, \ell_k)=\frac{1}{2\ell_k}\sum_{z=x}^{x+\ell_k-1}\left[\bar\eta_z-\bar\eta_{z+1}\right]\left[\frac{1-\sigma_z}{2}\right]
\end{equation*}
and $\sigma_z$ is as in \eqref{eq:sigma}. Then, the expectation in the statement is bounded from above by
\begin{align}
    &2\expected_{\nu_\rho^N}\left[\left(\int_0^t\sum_{x\in\Lambda_N} G\left(\frac{x}{N}\right)\varphi(\tau_x\eta_s)\sum_{y=1}^{\ell_k}\vec{\Sigma}_{x+y}(\eta_s, \ell_k)\de s\right)^2\right]\label{eq:Sigma_2block}
    \\&+2\expected_{\nu_\rho^N}\left[\left(\int_0^t\sum_{x\in\Lambda_N} G\left(\frac{x}{N}\right)\varphi(\tau_x\eta_s)\sum_{y=1}^{\ell_k}\vec{\Lambda}_{x+y}(\eta_s, \ell_k)\de s\right)^2\right].\label{eq:Lambda_2block}
\end{align}
By the Kipnis-Varadhan inequality and a standard convexity inequality, we have that
\begin{align}
    \eqref{eq:Sigma_2block}&\le Ct\ell_k\sum_{y=1}^{\ell_k}\left\|\sum_{x\in\Lambda_N} G\left(\frac{x}{N}\right)\varphi(\tau_x\eta)\vec{\Sigma}_{x+y}(\eta_s, \ell_k)\right\|_{-1}^2\nonumber
    \\\begin{split}&=Ct\ell_k\sum_{y=1}^{\ell_k}\mysup_{f\in L^2(\nu_\rho^N )}\bigg\{2\int_{\Omega_N}\sum_{x\in\Lambda_N} G\left(\frac{x}{N}\right)\varphi(\tau_x\eta)\vec{\Sigma}_{x+y}(\eta_s, \ell_k)f(\eta)\de \nu_\rho^N
    \\&\phantom{=}-N^a\mathscr{D}_N\left(f, \nu_\rho^N \right)\bigg\}.\label{eq:sup_2block}
    \end{split}
\end{align}
Now, we write
\begin{equation}\label{eq:S_twoB}
    \vec{\Sigma}_{x+y}(\eta, \ell_k)=\frac{1}{2\ell_k}\sum_{z=x+y}^{x+y+\ell_k-1}\left[\bar\eta_z-\bar\eta_{z+1}\right]-\vec{\Lambda}_{x+y}(\eta, \ell_k),
\end{equation}
and we write the integral term in \eqref{eq:sup_2block} corresponding to the sum in \eqref{eq:S_twoB} as
\begin{align}
    &2\int_{\Omega_N}\sum_{x\in\Lambda_N} G\left(\frac{x}{N}\right)\varphi(\tau_x\eta)\left(\frac{1}{2\ell_k}\sum_{z=y+x}^{y+x+\ell_k-1}\left[\bar\eta_z-\bar\eta_{z+1}\right]\right)f(\eta)\de \nu_\rho^N \nonumber
    \\&=\int_{\Omega_N}\sum_{x\in\Lambda_N} G\left(\frac{x}{N}\right)\varphi(\tau_x\eta)\left(\frac{1}{2\ell_k}\sum_{z=y+x}^{y+x+\ell_k-1}\left[\bar\eta_z-\bar\eta_{z+1}\right]\right)\left[f(\eta)-f(\eta^{z, z+1})\right]\de \nu_\rho^N  \label{eq:firstAdd_2block}
    \\&\phantom{=}+\int_{\Omega_N}\sum_{x\in\Lambda_N} G\left(\frac{x}{N}\right)\varphi(\tau_x\eta)\left(\frac{1}{2\ell_k}\sum_{z=y+x}^{y+x+\ell_k-1}\left[\bar\eta_z-\bar\eta_{z+1}\right]\right)\left[f(\eta)+f(\eta^{z, z+1})\right]\de \nu_\rho^N .\label{eq:secondAdd_2block} 
\end{align}
We can now apply Young's inequality as done in the proof of Lemma \ref{lemma:one_block} and use the same cancellations, so as to get that
\begin{align*}
    \eqref{eq:Sigma_2block}&\le \frac{Ct\ell_k}{N^{a-1}}\sum_{y=1}^{\ell_k}\|G\|_{2, N}^2 E_{\nu_\rho^N}\left[\varphi(\tau_x\eta)^2\right]
    \\&\le \frac{Ct\ell_k^2}{N^{a-1}}\|G\|_{2, N}^2\left(\frac{1}{\ell}+\frac{1}{N^\beta}\right).
\end{align*}
As for \eqref{eq:Lambda_2block}, by the same argument used to bound \eqref{eq:Lambda_1block} in the proof of Lemma \ref{lemma:one_block}:
\begin{enumerate}[i)]
    \item we get that
    \begin{equation*}
        \eqref{eq:Lambda_2block}\le Ct^2\ell_k^2\|G\|_{2, N}^2\left(\frac{\ell+\ell_{k+1}}{\ell N^{2\alpha-1}}+\frac{\ell+\ell_{k+1}}{N^{2\alpha+\beta-1}}+\frac{1}{N^{2\alpha+\beta-2}}+\frac{1}{N^{3\alpha-2}}\right);
    \end{equation*}

    \item if additionally \textbf{\textsc{Assumption 1a}} holds, then
    \begin{equation*}
        \eqref{eq:Lambda_2block}\le Ct\ell_k^2\|G\|_{2, N}^2\left(\frac{\ell+\ell_{k+1}}{\ell_k \ell N^{2\alpha-1}}+\frac{\ell+\ell_{k+1}}{\ell_k N^{2\alpha+\beta-1}}+\frac{1}{N^{2\alpha+\beta-2}}+\frac{1}{N^{3\alpha-2}}\right).
    \end{equation*}

\end{enumerate}
Combining everything together yields the required bound.
\end{proof}

The following corollary follows immediately from Lemma \ref{lemma:two_block}. The idea is simply to write the expression appearing in the statement as a telescopic sum involving boxes which double in size at every step, and then apply Lemma \ref{lemma:two_block} to each of these terms.

\begin{corollary}\label{corol:two_block} Fix $\ell_0, M\in\mathbb{N}$ and let $\varphi:\Omega_N\to\mathbb{R}$ be a function of the form $\varphi(\eta)=\cev{\eta}^\ell_0$ for some $1\le \ell\le N$.
\begin{enumerate}[i)]
    \item There exists a constant $C$ independent of $N$ such that, for any $t>0$ and any $G:\Lambda\to\mathbb{R}$ such that $\|G\|_{2, N}^2<\infty$,
    \begin{align*}
        &\expected_{\nu_\rho^N}\left[\left(\int_0^t\sum_{x\in\Lambda_N} G\left(\frac{x}{N}\right)\varphi(\tau_x\eta_s)\left[\vec\eta^{\ell_0}_{sN^a}(x+1)-\vec{\eta}^{2^{M-1}\ell_0}_{sN^a}(x)\right]\de s\right)^2\right]
        \\&\le Ct\ell_0^2\|G\|_{2, N}^2\left\{\frac{1}{N^{a-1}\ell}+\frac{1}{N^{a+\beta-1}}+\frac{t(\ell+\ell_0)}{N^{2\alpha-1}\ell}+\frac{t}{N^{2\alpha+\beta-2}}+\frac{t}{N^{3\alpha-2}}\right\};
    \end{align*}

    \item if, additionally, \textbf{\textsc{Assumption 1a}} holds, then the bound can be improved to
    \begin{equation*}
        Ct\ell_0^2\|G\|_{2, N}^2\left\{\frac{1}{N^{a-1}\ell}+\frac{1}{N^{a+\beta-1}}+\frac{t(\ell+\ell_0)}{N^{2\alpha-1}\ell_0\ell}+\frac{t}{N^{2\alpha+\beta-2}}+\frac{t}{N^{3\alpha-2}}\right\}.
    \end{equation*}

\end{enumerate}
The same result holds by substituting $\left[\vec\eta^{\ell_0}_{x+1}-\vec{\eta}^{2^{M-1}\ell_0}_x\right]$ with $\left[\cev\eta^{\ell_0}_{x+1}-\cev{\eta}^{2^{M-1}\ell_0}_x\right]$, as long as $\varphi$ is of the form $\varphi(\eta)=\vec{\eta}^\ell_0$ for some $1\le \ell\le N$.
\end{corollary}

%%%%%%%%%%%%%%%%%%%%%%%%%%%%%%%%%%%%%%%%%%%%%%%%%%
%%%Renormalisation Step%%%%%%%%%%%%%%%%%%%%%%%%%%%
%%%%%%%%%%%%%%%%%%%%%%%%%%%%%%%%%%%%%%%%%%%%%%%%%%
\subsection{Renormalisation Step}
The last preliminary result that we need in order to show Theorem \ref{thm:gen_bg} consists in the following lemma, which is known in the literature as the \textit{renormalisation step} as it involves a multi-scale analysis, which consists in replacing averages of particles in increasingly larger boxes. This lemma gives an upper bound on the price for one of these replacements, namely for going from a box  of size $\ell_0$ to a larger one of size $L$.
\begin{lemma}\label{lemma:renorm} The following holds. 
\begin{enumerate}[i)]

    \item There exists a constant $C$ independent of $N$ such that, for any $1\le \ell_0\le L$, any $t>0$ and any $G:\Lambda\to\mathbb{R}$ such that $\|G\|_{2, N}^2<\infty$,
    \begin{align*}
        &\expected_{\nu_\rho^N}\left[\left(\int_0^t\sum_{x\in\Lambda_N} G\left(\frac{x}{N}\right)\cev{\eta}^{\ell_0}_{sN^a}(x)\left[\vec\eta^{\ell_0}_{sN^a}(x)-\vec{\eta}^{L}_{sN^a}(x)\right]\de s\right)^2\right]\nonumber
        \\\begin{split} &\le Ct\|G\|_{2, N}^2L\left\{\frac{1}{N^{a-1}}+\frac{1}{N^{a+\beta-2}}+\frac{t}{N^{2\alpha-2}}+\frac{t}{N^{2\alpha+\beta-3}}+\frac{t}{N^{3\alpha-3}}\right\}
        \\&\phantom{\le}+Ct\|G\|_{2, N}^2\ell_0^2\left\{\frac{1}{N^{a-1}L}+\frac{1}{N^{a+\beta-1}}+\frac{t}{N^{2\alpha-1}}+\frac{t}{N^{2\alpha+\beta-2}}+\frac{t}{N^{3\alpha-2}}\right\};\end{split}
    \end{align*}
    
    \item if, additionally, \textbf{\textsc{Assumption 1a}} holds, then the last bound can be improved to 
    \begin{equation*}
        \begin{split} &Ct\|G\|_{2, N}^2 L \bigg\{\frac{1}{N^{a-1}}+\frac{1}{N^{a+\beta-2}}
        +\frac{t}{N^{2\alpha-1}}+\frac{t}{N^{2\alpha+\beta-3}}+\frac{t}{N^{3\alpha-3}}\bigg\}
        \\&+Ct\ell_0^2\|G\|_{2, N}^2\bigg\{\frac{1}{N^{a-1}L}+\frac{1}{N^{a+\beta-1}}+\frac{t}{\ell_0N^{2\alpha-1}}+\frac{t}{N^{2\alpha+\beta-2}}+\frac{t}{N^{3\alpha-2}}\bigg\}.
        \end{split}
    \end{equation*}
    
\end{enumerate}
\end{lemma}

\begin{proof}
We follow the proof of in \cite[Proposition 8]{gjs17}. First, assume that $L=2^M\ell_0$ for some $M\in\mathbb{N}$, and for each $k=0, \ldots, M-1$ let $\ell_{k+1}=2\ell_k$. We write
\begin{align*}
    \cev{\eta}^{\ell_0}_x\left[\vec\eta^{\ell_0}_x-\vec{\eta}^{L}_x\right]& = \sum_{k=0}^{M-1}\cev{\eta}^{\ell_k}_x\left[\vec\eta^{\ell_k}_x-\vec{\eta}^{\ell_{k+1}}_x\right]
    \\&\phantom{=}+\sum_{k=0}^{M-2}\vec{\eta}^{\ell_{k+1}}_x\left[\cev\eta^{\ell_k}_x-\cev{\eta}^{\ell_{k+1}}_x\right]
    \\&\phantom{=}+\vec{\eta}^{L}_x\left[\cev\eta^{\ell_{M-1}}_x-\cev{\eta}^{\ell_0}_x\right]
\end{align*}
Hence, by Minkowski's inequality and a convexity inequality, the expectation in the statement is bounded from above by
\begin{align}
    &C\sum_{k=0}^{M-1}\expected_{\nu_\rho^N}\left[\left(\int_0^t\sum_{x\in\Lambda_N}G\left(\frac{x}{N}\right)\cev{\eta}^{\ell_k}_{sN^a}(x)\left[\vec\eta^{\ell_k}_{sN^a}(x)-\vec{\eta}^{\ell_{k+1}}_{sN^a}(x)\right]\de s\right)^2\right]\label{eq:term_A}
    \\&+C\sum_{k=0}^{M-2}\expected_{\nu_\rho^N}\left[\left(\int_0^t\sum_{x\in\Lambda_N}G\left(\frac{x}{N}\right)\vec{\eta}^{\ell_{k+1}}_{sN^a}(x)\left[\cev\eta^{\ell_k}_{sN^a}(x)-\cev{\eta}^{\ell_{k+1}}_{sN^a}(x)\right]\de s\right)^2\right]\label{eq:term_B}
    \\&+C\expected_{\nu_\rho^N}\left[\left(\int_0^t\sum_{x\in\Lambda_N}G\left(\frac{x}{N}\right)\vec{\eta}^{\ell_{k+1}}_{sN^a}(x)\left[\cev\eta^{\ell_{M-1}}_{sN^a}(x)-\cev{\eta}^{\ell_0}_{sN^a}(x)\right]\de s\right)^2\right].\label{eq:term_C}
\end{align}
Now:
\begin{enumerate}[i)]
    
    \item by the two-block estimate with $\varphi(\eta)=\cev{\eta}^{\ell_k}_0$, we have that
    \begin{align*}
        \eqref{eq:term_A}&\le Ct\|G\|_{2, N}^2\hspace{-.3em}\sum_{k=1}^{M-1}\hspace{-.3em}\ell_k\left\{\frac{1}{N^{a-1}}+\frac{\ell_k}{N^{a+\beta-1}}+\frac{t\ell_k}{N^{2\alpha-1}}+\frac{t\ell_k}{N^{2\alpha+\beta-2}}+\frac{t\ell_k}{N^{3\alpha-2}}\right\}
        \\&\le Ct\|G\|_{2, N}^2 L\left\{\frac{1}{N^{a-1}}+\frac{1}{N^{a+\beta-2}}+\frac{t}{N^{2\alpha-2}}+\frac{t}{N^{2\alpha+\beta-3}}+\frac{t}{N^{3\alpha-3}}\right\},
    \end{align*}
    and similarly the same bound holds for \eqref{eq:term_B}. As for \eqref{eq:term_C}, by Corollary \ref{corol:two_block} with $\varphi(\eta)=\cev{\eta}^L_0$, we have that
    \begin{equation*}
        \eqref{eq:term_C}\le Ct\ell_0^2\|G\|_{2, N}^2\left\{\frac{1}{N^{a-1}L}+\frac{1}{N^{a+\beta-1}}+\frac{t}{N^{2\alpha-1}}+\frac{t}{N^{2\alpha+\beta-2}}+\frac{t}{N^{3\alpha-2}}\right\}
    \end{equation*}
    
    \item under \textbf{\textsc{Assumption 1a}}, again by the two-block estimate we have that
    \begin{align*}
        \eqref{eq:term_A}&\le Ct\|G\|_{2, N}^2\hspace{-.3em}\sum_{k=1}^{M-1}\hspace{-.3em}\ell_k\left\{\frac{1}{N^{a-1}}+\frac{\ell_k}{N^{a+\beta-1}}+\frac{t}{N^{2\alpha-1}}+\frac{t\ell_k}{N^{2\alpha+\beta-2}}+\frac{t\ell_k}{N^{3\alpha-2}}\right\}
        \\&\le CtL\|G\|_{2, N}^2 \bigg\{\frac{1}{N^{a-1}}+\frac{1}{N^{a+\beta-2}}+\frac{t}{N^{2\alpha+\beta-3}}+\frac{t}{N^{2\alpha-1}}+\frac{t}{N^{3\alpha-3}}\bigg\}
    \end{align*}
    and the same bound holds for \eqref{eq:term_B}, whereas, by Corollary \ref{corol:two_block}, 
    \begin{equation*}
        \eqref{eq:term_C}\le Ct\ell_0^2\|G\|_{2, N}^2\left\{\frac{1}{N^{a-1}L}+\frac{1}{N^{a+\beta-1}}+\frac{t}{N^{2\alpha-1}\ell_0}+\frac{t}{N^{2\alpha+\beta-2}}+\frac{t}{N^{3\alpha-2}}\right\}.
    \end{equation*}

\end{enumerate}
Combining everything together finally yields the claimed bounds. If instead $L$ cannot be written as $2^M\ell_0$, it suffices to choose $M$ such that $2^M\ell_0<L<2^{M+1}\ell_0$, and a similar computation yields the same bounds.
\end{proof}

%%%%%%%%%%%%%%%%%%%%%%%%%%%%%%%%%%%%%%%%%%%%%%%%%%
%%%Proof of Main Theorem%%%%%%%%%%%%%%%%%%%%%%%%%%
%%%%%%%%%%%%%%%%%%%%%%%%%%%%%%%%%%%%%%%%%%%%%%%%%%
\subsection{Proof of Theorem \ref{thm:gen_bg}}
We are now ready to show the generalised second-order Boltzmann-Gibbs principle. 

\begin{proof}[Proof of Theorem \ref{thm:gen_bg}] Let $\ell_0\le L$. We write $\bar\eta_x\bar\eta_{x+1}-\cev{\eta}_x^L\vec{\eta}_x^L$ in the expectation appearing in the statement as
\begin{align}
    \bar\eta_x\bar\eta_{x+1}-\cev{\eta}_x^L\vec{\eta}_x^L&=\bar{\eta}_x\left(\bar\eta_{x+1}-\vec{\eta}^{\ell_0}_x\right)\label{eq:term_one}
    \\&\phantom{=}+\vec{\eta}^{\ell_0}_x\left(\bar\eta_x-\cev{\eta}^{\ell_0}_x\right)\label{eq:term_two}  \\&\phantom{=}+\cev{\eta}^{\ell_0}_x\left(\vec{\eta}^{\ell_0}_x-\vec{\eta}^L_x\right)\label{eq:term_three}
    \\&\phantom{=}+\vec{\eta}^L_x\left(\cev{\eta}^{\ell_0}_x-\bar{\eta}_x\right)\label{eq:term_four}
    \\&\phantom{=}+\vec{\eta}^L_x\left(\bar\eta_x-\cev{\eta}^L_x\right).\label{eq:term_five}
\end{align}
Then:
\begin{enumerate}[i)]
    \item the four terms \eqref{eq:term_one}, \eqref{eq:term_two}, \eqref{eq:term_four} and \eqref{eq:term_five} can all be treated using the one-block estimate. Indeed, by applying Lemma \ref{lemma:one_block} with $\varphi(\eta)=\eta_0$, $\varphi(\eta)=\vec{\eta}^{\ell_0}_0$ and twice with $\varphi(\eta)=\vec{\eta}^{L}_0$ respectively, we have that:
    \begin{align*}
        &\expected_{\nu_\rho^N}\left[\left(\int_0^t\sum_{x\in\Lambda_N}G\left(\frac{x}{N}\right)\bar{\eta}_{sN^a}(x)\left[\bar\eta_{sN^a}(x+1)-\vec{\eta}^{\ell_0}_{sN^a}(x)\right]\de s\right)^2\right]
        \\&\le Ct\|G\|_{2, N}^2\left\{\frac{\ell_0^2}{N^{a-1}}+\frac{\ell_0^2}{N^{a+\beta-1}}+\frac{t\ell_0^3}{N^{2\alpha-1}}+\frac{t\ell_0^2}{N^{2\alpha+\beta-2}}+\frac{t\ell_0^2}{N^{3\alpha-2}}\right\},
    \end{align*}
    \begin{align*}
        &\expected_{\nu_\rho^N}\left[\left(\int_0^t\sum_{x\in\Lambda_N}G\left(\frac{x}{N}\right)\vec{\eta}^{\ell_0}_{sN^a}(x)\left[\bar\eta_{sN^a}(x)-\cev{\eta}^{\ell_0}_{sN^a}(x)\right]\de s\right)^2\right]
        \\&\le Ct\|G\|_{2, N}^2\left\{\frac{\ell_0}{N^{a-1}}+\frac{\ell_0^2}{N^{a+\beta-1}}+\frac{t\ell_0^2}{N^{2\alpha-1}}+\frac{t\ell_0^2}{N^{2\alpha+\beta-2}}+\frac{t\ell_0^2}{N^{3\alpha-2}}\right\},
    \end{align*}
    \begin{align*}
        &\expected_{\nu_\rho^N}\left[\left(\int_0^t\sum_{x\in\Lambda_N}G\left(\frac{x}{N}\right)\vec{\eta}^L_{sN^a}(x)\left[\cev\eta_{sN^a}^{\ell_0}(x)-\bar{\eta}_{sN^a}(x)\right]\de s\right)^2\right]
        \\&\le Ct\|G\|_{2, N}^2\left\{\frac{\ell_0^2}{N^{a-1}L}+\frac{\ell_0^2}{N^{a+\beta-1}}+\frac{t\ell_0^2}{N^{2\alpha-1}}+\frac{t\ell_0^2}{N^{2\alpha+\beta-2}}+\frac{t\ell_0^2}{N^{3\alpha-2}}\right\},
    \end{align*}
    and
    \begin{align*}
        &\expected_{\nu_\rho^N}\left[\left(\int_0^t\sum_{x\in\Lambda_N}G\left(\frac{x}{N}\right)\vec{\eta}^L_{sN^a}(x)\left[\bar\eta_{sN^a}(x)-\cev{\eta}^L_{sN^a}(x)\right]\de s\right)^2\right]
        \\&\le Ct\|G\|_{2, N}^2\left\{\frac{L}{N^{a-1}}+\frac{L^2}{N^{a+\beta-1}}+\frac{tL^2}{N^{2\alpha-1}}+\frac{tL^2}{N^{2\alpha+\beta-2}}+\frac{tL^2}{N^{3\alpha-2}}\right\}.
    \end{align*}
    As for \eqref{eq:term_three}, by Lemma \ref{lemma:renorm} we have that
    \begin{align*}
        &\expected_{\nu_\rho^N}\left[\left(\int_0^t\sum_{x\in\Lambda_N}G\left(\frac{x}{N}\right)\cev{\eta}^{\ell_0}_{sN^a}(x)\left[\vec{\eta}^{\ell_0}_{sN^a}(x+1)-\vec{\eta}^L_{sN^a}(x)\right]\de s\right)^2\right]
        \\\begin{split} &\le Ct\|G\|_{2, N}^2\left\{\frac{L}{N^{a-1}}+\frac{L}{N^{a+\beta-2}}+\frac{tL}{N^{2\alpha-2}}+\frac{tL}{N^{2\alpha+\beta-3}}+\frac{tL}{N^{3\alpha-3}}\right\}
        \\&\phantom{\le}+Ct\|G\|_{2, N}^2\left\{\frac{\ell_0^2}{N^{a-1}L}+\frac{\ell_0^2}{N^{a+\beta-1}}+\frac{t\ell_0^2}{N^{2\alpha-1}}+\frac{t\ell_0^2}{N^{2\alpha+\beta-2}}+\frac{t\ell_0^2}{N^{3\alpha-2}}\right\}.\end{split}
    \end{align*}
    Since $\ell_0$ is arbitrary, choosing for instance $\ell_0=L^{\frac{1}{3}}$ and combining everything together yields \eqref{eq:bg_bound};
    
    \item the bound \eqref{eq:bg_boundA} follows by the exact same argument using the bounds obtained under \textbf{\textsc{Assumption 1a}} instead. 

\end{enumerate}
This completes the proof.
\end{proof}

%%%%%%%%%%%%%%%%%%%%%%%%%%%%%%%%%%%%%%%%%%%%%%%%%%
%%%EQUILIBIRUM FLUCTUATIONS OF THE KLS MODEL%%%%%%
%%%%%%%%%%%%%%%%%%%%%%%%%%%%%%%%%%%%%%%%%%%%%%%%%%
\section{Equilibrium Fluctuations of the KLS Model}\label{sec:kls}
In this section we show Theorem \ref{thm:fluctuations}. Our approach follows the classical scheme: one first writes the discrete martingale equation for the density fluctuation field $\{\mathscr{Y}_t^N, t\in[0, T]\}$ and shows that the sequence of measures $\{\mathcal{Q}_N\}_N$ is tight in $\mathcal{D}([0, T], \mathfrak{D}'(\mathbb{T}))$ with respect to the Skorohod topology, and then proceeds to characterise its limit point $\mathcal{Q}$ as the law of an appropriate equation. As the tightness argument is standard, we only focus on the characterisation of the limit point and especially on the quadratic term in the martingale equation, which is the one that calls for Theorem \ref{thm:gen_bg}. Note, however, that as discussed in Section \ref{subsec:dynamics_measure}, the model does \textit{not} satisfy \textbf{\textsc{Assumption 3}} of the generalised second-order Boltzmann–Gibbs principle, since the rates may be degenerate. We will therefore show that the analysis can be restricted to a suitable set of configurations on which the assumption does hold.

Throughout, we will denote by $\mathbb{P}_{\nu_N^\alpha}$ the probability measure on the space of càdlàg trajectories from $[0, T]$ to $\Omega_N$ corresponding to the process $\{\eta_t^N, t\in[0, T]\}$ started at its stationary measure $\nu_N^\alpha$, and by $\expected_{\nu_N^\alpha}[\,\cdot\,]$ expectations with respect to $\mathbb{P}_{\nu_N^\alpha}$.

%%%%%%%%%%%%%%%%%%%%%%%%%%%%%%%%%%%%%%%%%%%%%%%%%%
%%%Properties of the Invariant Measure%%%%%%%%%%%%
%%%%%%%%%%%%%%%%%%%%%%%%%%%%%%%%%%%%%%%%%%%%%%%%%%
\subsection{Properties of the Invariant Measure}
We state here some results about the stationary measure $\nu_N^\alpha$, introduced in \eqref{eq:nuN_alpha}, whose proofs are postponed to Appendix \ref{sec:appendix}.

The first one regards the asymptotic behaviour of the partition function \eqref{eq:partition}:

\begin{lemma}\label{lemma:partition}
If $\alpha>1$, then the partition function $Z_N^\alpha$ satisfies
\begin{equation*}
    \mylim_{N\to\infty} \frac{Z_N^\alpha}{2^N}=1.
\end{equation*}
\end{lemma}

We will also be using the following facts about the asymptotic average of occupation variables and about the correlation functions. Note that, unlike Lemma \ref{lemma:partition}, these do not require the restriction $\alpha>1$.

\begin{lemma}\label{lemma:expectation} For any $x\in\mathbb{T}_N$,
\begin{equation*}
   \left|E_{\nu_N^\alpha}[\eta_x]-\frac{1}{2}\right|\lesssim N^{-\alpha}.
\end{equation*}
In particular, $\mylim_{N\to\infty} \bar\rho_N=\frac{1}{2}$.
\end{lemma}

\begin{lemma}\label{lemma:correlations}
Let $m\ge 2$ and, given $x_1<\ldots<x_m\in\mathbb{T}_N$, let $d_j:=x_{j+1}-x_j$ for each $j=1, \ldots, m$, where $x_{m+1}:=x_1+N$ (so that, in particular, $d_1+\ldots+d_m=N$). Then,
\begin{equation}\label{eq:correlation_bound}
    \left|E_{\nu_N^\alpha}\left[\prod_{j=1}^{m} \bar\eta_{x_j}\right]\right|\lesssim N^{-\alpha d},
\end{equation}
where $d:=\mymin\{d_1, \ldots, d_m\}$.
\end{lemma}

We highlight that, in the result above, the number of sites $m$ does not appear on the right-hand side of \eqref{eq:correlation_bound}, as the bound only depends on the minimum distance between any pair of considered points.

Finally, the result that follows guarantees that any limit point $\{\mathscr{Y}_t, t\in[0, T]\}$ of the sequence of fields $\{\mathscr{Y}_t^N, t\in[0, T]\}_N$ satisfies item i) of Definition \ref{def:energy_solutions}.

\begin{proposition}\label{prop:fixed_time}
Let $g:\mathbb{T}\to\mathbb{R}$ be a continuous map: for any $\alpha>0$, under the stationary measure $\nu_N^\alpha$, the random variable
\begin{equation*}
    \mathscr{X}^N:=\frac{1}{\sqrt{N}}\sum_{x\in\mathbb{T}_N} g\Bigl(\frac{x}{N}\Bigr)\bar\eta_x
\end{equation*}
converges in distribution to a mean-zero Gaussian of variance $\frac{1}{4}\int_{\mathbb{T}}g(u)^2\de u$.
\end{proposition}

%%%%%%%%%%%%%%%%%%%%%%%%%%%%%%%%%%%%%%%%%%%%%%%%%%
%%%Dynkin's Martingales%%%%%%%%%%%%%%%%%%%%%%%%%%%
%%%%%%%%%%%%%%%%%%%%%%%%%%%%%%%%%%%%%%%%%%%%%%%%%%
\subsection{Dynkin's Martingales}\label{sec:dynkin}
By Dynkin's formula (see for example \cite[Appendix 1]{kl99}), for any $\phi\in\mathfrak{D}(\mathbb{T})$, the processes $\{\mathscr{M}^N_t(\phi), t\in[0, T]\}$ and $\{\mathscr{N}^N_t(\phi), t\in[0, T]\}$ defined via
\begin{align}
    &\mathscr{M}_t^N(\phi):=\mathscr{Y}_t^N(\phi)-\mathscr{Y}_0^N(\phi)-\int_0^t(N^2\mathcal{L}^N+\partial_s)\mathscr{Y}_s^N(\phi)\de s,\label{eq:dynkin_formula}
    \\&\mathscr{N}_t^N(\phi):=\mathscr{M}_t^N(\phi)^2-\int_0^t \left\{N^2\mathcal{L}^N\mathscr{Y}_s^N(\phi)^2-2N^2\mathscr{Y}_s^N(\phi)\mathcal{L}^N\mathscr{Y}_s^N(\phi)\right\}\de s \label{eq:quad_var}
\end{align}
are both martingales with respect to the natural filtration of $\{\mathscr{Y}_t^N(\phi), t\in[0, T]\}$. 

A simple computation shows that, for each $\phi\in\mathfrak{D}(\mathbb{T})$, the predictable quadratic variation of $\{\mathscr{M}_t^N(\phi), t\in[0, T]\}$ in \eqref{eq:quad_var} can be written as
\begin{equation*}
    \left\langle\mathscr{M}^N(\phi)\right\rangle_t=\int_0^t\frac{1}{N}\sum_{x\in\mathbb{T}_N}c_{x, x+1}^N(\eta_{sN^2})[\eta_{sN^2}(x)-\eta_{sN^2}(x+1)]^2\nabla_N\phi\left(\frac{x}{N}\right)^2\de s.
\end{equation*}
Recalling the choice of rates
\begin{equation}\label{eq:our_rates}
    \begin{split}c_{x,x+1}^N(\eta)&=\Bigl(\frac{1}{2}+\frac{b}{2N^\gamma}\Bigr)\eta_x(1-\eta_{x+1})\{(\eta_{x-1}+\eta_{x+2})+\eps_N(\eta_{x-1}-\eta_{x+2})\}
    \\&\phantom{=}+\Bigl(\frac{1}{2}-\frac{b}{2N^\gamma}\Bigr)\eta_{x+1}(1-\eta_x)\{(\eta_{x-1}+\eta_{x+2})-\eps_N(\eta_{x-1}-\eta_{x+2})\}
    \end{split}
\end{equation}
and using Lemmas \ref{lemma:expectation} and \ref{lemma:correlations}, it is immediate to see that
\begin{equation*}
    \mylim_{N\to\infty}\expected_{\nu_N^\alpha}\left[\left\langle\mathscr{M}^N(\phi)\right\rangle_t\right]=\mylim_{N\to\infty}2t\bar\rho_N^2(1-\bar\rho_N)\frac{1}{N}\sum_{x\in\mathbb{T}_N}\nabla_N\phi\left(\frac{x}{N}\right)^2=\frac{t}{4}\|\nabla\phi\|^2_{L^2(\mathbb{T})}.
\end{equation*}
By a standard argument (see for example \cite[Theorem VIII.3.11]{js03}), this implies that $\{\mathscr{M}_t^N, t\in[0, T]\}_N$ converges in distribution in $\mathcal{D}([0, T], \mathfrak{D}'(\mathbb{T}))$ to $\frac{1}{2}\nabla\dot{\mathscr{W}}_t$, with $\{\mathscr{W}_t, t\in[0, T]\}$ a $\mathfrak{D}'(\mathbb{T})$-valued Brownian motion. Moreover one can check that, defining
\begin{equation*}
    h(\eta):=\eta_0\eta_1+\eta_0\eta_{-1}-\eta_{-1}\eta_1
\end{equation*}
and
\begin{equation*}
    g(\eta):=\eta_0\eta_{-1}+\eta_0\eta_1+\eta_{-1}\eta_1-2\eta_{-1}\eta_0\eta_1,
\end{equation*} 
the integral term in \eqref{eq:dynkin_formula} can be written as 
\begin{align}
    \int_0^t&(N^2\mathcal{L}^N+\partial_s)\mathscr{Y}_s^N(\phi)\de s=-\frac{\boldsymbol{v}_N}{N\sqrt{N}}\int_0^t\sum_{x\in\mathbb{T}_N}\phi'\Bigl(\frac{x-\boldsymbol{v}_N s}{N}\Bigr)\bar{\eta}_{sN^2}(x)\de s\label{eq:velocity_term}
    \\&+\frac{1}{2\sqrt{N}}\int_0^t \sum_{x\in\mathbb{T}_N}\Delta_N\phi\Bigl(\frac{x-\boldsymbol{v}_N s}{N}\Bigr)h(\tau_x\eta_{sN^2})\de s\label{eq:h_term}
    \\&+(4b\bar\rho_N-6b\bar\rho_N^2) N^{\frac{1}{2}-\gamma}\int_0^t\sum_{x\in\mathbb{T}_N}\nabla_N\phi\Bigl(\frac{x-\boldsymbol{v}_N s}{N}\Bigr)\bar\eta_{sN^2}(x)\de s\label{eq:linear_term}
    \\\begin{split}&+(b-3b\bar\rho_N)N^{\frac{1}{2}-\gamma}\int_0^t\sum_{x\in\mathbb{T}_N}\nabla_N\phi\Bigl(\frac{x-\boldsymbol{v}_N s}{N}\Bigr)\times\\&\phantom{+(b-3b\bar\rho_N)N^{\frac{1}{2}-\gamma}}\times[\bar\eta_{sN^2}(x)\bar\eta_{sN^2}(x+1)+\bar\eta_{sN^2}(x)\bar\eta_{sN^2}(x+2)]\de s\end{split}\label{eq:quadratic_term}
    \\&-2b N^{\frac{1}{2}-\gamma}\int_0^t\sum_{x\in\mathbb{T}_N}\nabla_N\phi\Bigl(\frac{x-\boldsymbol{v}_N s}{N}\Bigr)[\bar\eta_{sN^2}(x)\bar\eta_{sN^2}(x+1)\bar\eta_{sN^2}(x+2)]\de s\label{eq:cubic_term}
    \\\begin{split}&+\frac{\eps_N}{2\sqrt{N}}\int_0^t\sum_{x\in\mathbb{T}_N}\Delta_N\phi\Bigl(\frac{x-\boldsymbol{v}_N s}{N}\Bigr)g(\tau_x\eta_{sN^2})\de s
    \\&+b\eps_N N^{\frac{1}{2}-\gamma}\int_0^t\sum_{x\in\mathbb{T}_N}\nabla_N\phi\Bigl(\frac{x-\boldsymbol{v}_N s}{N}\Bigr)\eta_{sN^2}(x)[\eta_{sN^2}(x+1)-\eta_{sN^2}(x+2)]\de s
    \end{split}\nonumber
\end{align}
plus a term which vanishes in $L^2(\prob_{\nu_N^\alpha})$ as $N\to\infty$.

Since $|\eps_N|\lesssim N^{-\alpha}$ with $\alpha>1$, it is immediate to see that the two bottom lines of the equation above vanish in $L^2(\prob_{\nu_N^\alpha})$ as $N\to\infty$. Now, by the choice of the transport velocity $\boldsymbol{v_N}$ in \eqref{eq:transport_velocity} and a Taylor expansion of $\phi$, we see that $\eqref{eq:velocity_term}+\eqref{eq:linear_term}$ vanishes in $L^2(\prob_{\nu_N^\alpha})$ as $N\to\infty$. In fact, the transport velocity was chosen exactly so as to ``kill" these two degree-one terms, which would otherwise diverge. The linear term \eqref{eq:h_term} can be treated as in \cite[Section 4a]{bgs17}, and in particular it can be rewritten as
\begin{equation*}
    -\int_0^t\frac{1}{2\sqrt{N}}\sum_{x\in\mathbb{T}_N}2\bar\rho_N\Delta_N \phi\Bigl(\frac{x}{N}\Bigr)\bar\eta_{sN^2}(x)\de s=-\bar\rho_N\int_0^t\mathscr{Y}_s^N(\Delta_N\phi)\de s
\end{equation*}
plus a term which vanishes in $L^2(\prob_{\nu_N^\alpha})$. Using Lemma \ref{lemma:expectation}, it is standard to conclude that this yields the diffusive contribution in \eqref{eq:limit_OU} or \eqref{eq:limit_SBE}. Finally, once the quadratic term has been handled, the cubic term \eqref{eq:cubic_term} can be treated in the same way as in \cite[Lemma 6.2]{bgs17}, and in particular, one shows that it vanishes in $L^2(\prob_{\nu_N^\alpha})$ as $N\to\infty$. Indeed, in light of Remark \ref{remark:varphi}, all the lemmas in Section \ref{sec:proof_gen_bg} extend to more general functions $\varphi$, provided they are of local polynomial form and satisfy the support condition. For a detailed treatment we refer the reader to \cite{bgs17}, and omit the proof here.

The most challenging term is indeed the quadratic one, namely \eqref{eq:quadratic_term}, which we will treat in the two sections that follow.

%%%%%%%%%%%%%%%%%%%%%%%%%%%%%%%%%%%%%%%%%%%%%%%%%%
%%%Restriction of Configurations%%%%%%%%%%%%%%%%%%
%%%%%%%%%%%%%%%%%%%%%%%%%%%%%%%%%%%%%%%%%%%%%%%%%%
\subsection{Restriction of Configurations}
As previously noted and similarly to the porous media model, \textbf{\textsc{Assumption 3}} of Theorem \ref{thm:gen_bg} does not hold because of the potential degeneracy of the rates, so we first need to show that we can restrict \eqref{eq:quadratic_term} to an appropriate set of configurations. The idea is to split terms taking into account sets that have at least two particles at a distance at most two -- what we call a \textit{mobile cluster} -- and their complement. We will call the former \textit{good configurations}, namely configurations that have at least one mobile cluster, and the latter \textit{bad configurations}.

We start by noting that, from Lemma \ref{lemma:partition}, we can deduce the following straightforward corollary. Following \cite[Section 2.3]{bgs17}, fix $x\in\mathbb{T}_N$ and an integer $1\le \ell\le N$: we define the set of \textit{good configurations} inside the box $\Lambda_x^\ell:=\{x+1, \ldots, x+\ell\}$ of size $\ell$ to the right of $x$ as 
\begin{equation}\label{eq:good_configs}
    \mathscr{G}_\ell(x)=\left\{\eta\in\Omega_N: \sum_{y=x}^{x+\ell-2} \eta_y(\eta_{y+1}+\eta_{y+2})+\eta_{x+\ell-1}\eta_{x+\ell}>0\right\},
\end{equation}
as well as its complement of \textit{bad configurations}  as
\begin{equation}\label{eq:bad_configs}
    \mathscr{B}_\ell(x):=\mathscr{G}_\ell(x)^c=\left\{\eta: \sum_{y=x}^{x+\ell-2} \eta_y(\eta_{y+1}+\eta_{y+2})+\eta_{x+\ell-1}\eta_{x+\ell}=0\right\}.
\end{equation}

\begin{corollary}\label{corol:bad_measure}
If $\alpha>1$, for each $x\in\mathbb{T_N}$, the $\nu_N^\alpha$-measure of the bad configurations $\mathscr{B}_\ell(x)$ decays exponentially in $\ell$, and in particular
\begin{equation}
    \nu_N^\alpha(\mathscr{B}_\ell(x))\lesssim 2^{-\lfloor\frac{\ell}{3}\rfloor}.
\end{equation}
\end{corollary}

\begin{proof}
Let $\mu_{\frac{1}{2}}$ denote the Bernoulli product measure on $\{0, 1\}^{\mathbb{T}_N}$ of parameter $\frac{1}{2}$, namely $\mu_{\frac{1}{2}}\{\eta_x=1\}=\frac{1}{2}$. Note that, if $\eta\in\mathscr{B}_\ell(x)$, then any two particles in the box $\Lambda_x^\ell$ are at least at distance $3$, so that, if we divide $\Lambda_x^\ell$ into $\left\lfloor\frac{\ell}{3}\right\rfloor$ consecutive segments of size $3$, each of these has at most one particle, and thus
\begin{equation*}
    \mu_{\frac{1}{2}}(\mathscr{B}_\ell(x))\le \left(4\cdot\frac{1}{2^3}\right)^{\lfloor\frac{\ell}{3}\rfloor}.
\end{equation*}
But now, by Lemma \ref{lemma:partition}, we have that
\begin{equation*}
    \nu_N^\alpha(\mathscr{B}_\ell(x))=\frac{1}{Z_N^\alpha}\sum_{\eta\in\mathscr{B}_\ell(x)}e^{-N^{-\alpha} H(\eta)}\lesssim \frac{|\mathscr{B}_\ell(x)|}{2^N}=\mu_{\frac{1}{2}}(\mathscr{B}_\ell(x)),
\end{equation*}
which completes the proof.
\end{proof}

Corollary \ref{corol:bad_measure} easily implies that we can discard a set of configurations $\eta$ where the rates $c_{x, x+1}^N(\eta)$ are potentially degenerate, which is the content of the next result.

\begin{lemma}[Restriction of Configurations]\label{lemma:restriction}
Let $\varphi:\Omega_N\to\mathbb{R}$. There exists a constant $C$ such that, for any $t>0$, any positive integer $N$ and any positive integers $\ell_0, \ell \le N$,
\begin{equation}\label{eq:restriction_bound}
    \begin{split}
    \expected_{\nu_N^\alpha}&\left[\left(\int_0^t\sum_{x\in\mathbb{T}_N}G\left(\frac{x}{N}\right)\varphi(\tau_x\eta_{sN^2})\de s\right)^2\right] \le Ct^2 N^2 2^{-\lfloor\frac{\ell_0}{3}\rfloor}\|G\|_{2, N}^2\mysup_{\eta\in\Omega_N}\varphi(\eta)^2    \\&+2\expected_{\nu_N^\alpha}\left[\left(\int_0^t\sum_{x\in\mathbb{T}_N}G\left(\frac{x}{N}\right)\varphi(\tau_x\eta_{sN^2})\boldsymbol{1}_{\mathscr{G}_{\ell_0}(x+\ell)}(\eta_{sN^2})\de s\right)^2\right].
    \end{split}
\end{equation}
\end{lemma}

\begin{proof} Recall the definition of good and bad configurations given in \eqref{eq:good_configs} and \eqref{eq:bad_configs}: by the inequality $(x+y)^2\le 2x^2+2y^2$, for any integers $1\le \ell_0, \ell \le N$ we can bound the expectation in the statement from above by
\begin{align}
    &2\expected_{\nu_N^\alpha}\left[\left(\int_0^t\sum_{x\in\mathbb{T}_N}G\left(\frac{x}{N}\right)\varphi(\tau_x\eta_{sN^2})\boldsymbol{1}_{\mathscr{G}_{\ell_0}(x+\ell)}(\eta_{sN^2})\de s\right)^2\right]\nonumber
    \\&+2\expected_{\nu_N^\alpha}\left[\left(\int_0^t\sum_{x\in\mathbb{T}_N}G\left(\frac{x}{N}\right)\varphi(\tau_x\eta_{sN^2})\boldsymbol{1}_{\mathscr{B}_{\ell_0}(x+\ell)}(\eta_{sN^2})\de s\right)^2\right].\label{eq:bad_term}
\end{align}

Consider \eqref{eq:bad_term}: since the measure $\nu_N^\alpha$ is invariant for the dynamics, by the Cauchy-Schwarz inequality and a standard convexity inequality, we have that
\begin{align*}
    &\expected_{\nu_N^\alpha}\left[\left(\int_0^t\sum_{x\in\mathbb{T}_N}G\left(\frac{x}{N}\right)\varphi(\tau_x\eta_{sN^2})\boldsymbol{1}_{\mathscr{B}_{\ell_0}(x+\ell)}(\eta_{sN^2})\de s\right)^2\right]
    \\& \le t^2 E_{\nu_N^\alpha}\left[\left(\sum_{x\in\mathbb{T}_N}G\left(\frac{x}{N}\right)\varphi(\tau_x\eta)\boldsymbol{1}_{\mathscr{B}_{\ell_0}(x+\ell)}(\eta)\de s\right)^2\right]
    \\&\le t^2N \sum_{x\in\mathbb{T}_N} G\left(\frac{x}{N}\right)^2 E_{\nu_N^\alpha}\left[\varphi(\tau_x\eta)^2\boldsymbol{1}_{\mathscr{B}_{\ell_0}(x+\ell)}(\eta)\right]
    \\&\le t^2N^2\|G\|^2_{2, N}\mysup_{\eta\in\Omega_N}\varphi(\eta)^2\mysup_{x\in\mathbb{T}_N}\nu_N^\alpha(\mathscr{B}_{\ell_0}(x+\ell)).
\end{align*}
By Corollary \ref{corol:bad_measure}, we are done.
\end{proof}

Finally, the next result is an adaptation of \cite[Lemma 5.7]{bgs17} and demonstrates how the presence of a ``good box'' lets us in fact get around the potential degeneracy of the rates, and more specifically, allows us to repeat all the arguments in Section \ref{sec:proof_gen_bg} as long as we restrict to the presence of such a box. 

\begin{lemma}[Path Lemma]\label{lemma:path}
Given $y, z\in\mathbb{T}_N$, define
\begin{equation*}
    \sigma_{y, z}=\sigma_{y, z}(\eta):=\frac{\nu_N^\alpha(\eta^{y, z})}{\nu_N^\alpha(\eta)}.
\end{equation*}
For any $f\in L^2(\nu_N^\alpha)$, any positive integers $N$ and $\ell_0, \ell\le \frac{N}{2}$, any $y, z\in\{1, \ldots, \ell\}$ and {any mean-zero local function $\varphi:\Omega_N\to\mathbb{R}$ such that $\varphi(\eta^{y, z})=\varphi(\eta)$} for each $\eta\in\Omega_N$, we have that
\begin{align*}
    &\left|2\int_{\Omega_N} \sum_{x\in\mathbb{T}_N} G\left(\frac{x}{N}\right)\varphi(\tau_x\eta)[\eta_{x+y}-\eta_{x+z}]\boldsymbol{1}_{\mathscr{G}_{\ell_0}(x+\ell)}(\eta)f(\eta) \left[\frac{1+\sigma_{x+y, x+z}}{2}\right]\de \nu_N^\alpha\right|
    \\& \le C\frac{(\ell+\ell_0)^2}{N}\Var_{\nu_N^\alpha}(\varphi)\|G\|_{2, N}^2+N^2\mathscr{D}_N(f, \nu_N^\alpha).
\end{align*}
The same estimate holds by substituting $\boldsymbol{1}_{\mathscr{G}_{\ell_0}(x+\ell)}$ with $\boldsymbol{1}_{\mathscr{G}_{\ell_0}(x-\ell)}$, {as long as $y$ and $z$ are instead in the set $\{-\ell, \ldots, -1\}$.}
\end{lemma}

\begin{remark} The Dirichlet form $\mathscr{D}_N$ was defined in \eqref{eq:def_dirichlet}. Using the stationarity of $\nu_N^\alpha$, it is not hard to see that
\begin{equation}\label{eq:dirichlet_KLS}
    \mathscr{D}_N(f, \nu_N^\alpha)=\frac{1}{2}\int_{\Omega_N}\sum_{x\in\mathbb{T}_N}c_{x, x+1}^N(\eta)[f(\eta^{x, x+1})-f(\eta)]\de\nu_N^\alpha.
\end{equation}
\end{remark}

\begin{proof}[Proof of Lemma \ref{lemma:path}] We start by noting that the integral in the statement can be rewritten as
\begin{align}
    &\int_{\Omega_N} \sum_{x\in\mathbb{T}_N} G\left(\frac{x}{N}\right)\varphi(\tau_x\eta)[\eta_{x+y}-\eta_{x+z}]\boldsymbol{1}_{\mathscr{G}_{\ell_0}(x+\ell)}(\eta)[f(\eta)-f(\eta^{x+y, x+z})] \de \nu_N^\alpha\nonumber
    \\&+\int_{\Omega_N} \sum_{x\in\mathbb{T}_N} G\left(\frac{x}{N}\right)\varphi(\tau_x\eta)[\eta_{x+y}-\eta_{x+z}]\boldsymbol{1}_{\mathscr{G}_{\ell_0}(x+\ell)}(\eta)[f(\eta)+f(\eta^{x+y, x+z})]\de \nu_N^\alpha\label{eq:plus_term}
    \\&+\int_{\Omega_N} \sum_{x\in\mathbb{T}_N} G\left(\frac{x}{N}\right)\varphi(\tau_x\eta)[\eta_{x+y}-\eta_{x+z}]\boldsymbol{1}_{\mathscr{G}_{\ell_0}(x+\ell)}(\eta)f(\eta) \left[\sigma_{x+y, x+z}-1\right]\de \nu_N^\alpha\label{eq:sigma_term},
\end{align}
and by making the change of variable $\eta^{x+y, x+z}\mapsto \eta$ in the second summand of \eqref{eq:plus_term}, recalling the assumption $\varphi((\tau_x\eta)^{x+y, x+z})=\varphi(\tau_x\eta)$, it is easy to see that $\eqref{eq:plus_term}+\eqref{eq:sigma_term}=0$. Hence, the expression in the statement is bounded from above by
\begin{equation*}
    \sum_{x\in\mathbb{T}_N} \left|G\left(\frac{x}{N}\right)\right| \left|\int_{\Omega_N}\varphi(\tau_x\eta)[\eta_{x+y}-\eta_{x+z}]\boldsymbol{1}_{\mathscr{G}_{\ell_0}(x+\ell)}(\eta)[f(\eta)-f(\eta^{x+y, x+z})] \de \nu_N^\alpha\right|.
\end{equation*}

Now given $\eta\in\mathscr{G}_{\ell_0}(x+\ell)$, let $X(\eta):=\mymin\{x'\in\Lambda_{x+\ell}^{\ell_0}:\eta_{x'}\eta_{x'+1}+\eta_{x'}\eta_{x'+2}>0\}$. For all configurations $\eta$ such that $X(\eta)=x'$, we can construct a sequence $\{x^{(i)}\}_{i=0, \ldots, N(x')}$ taking values in $\{x+1, \ldots, x+\ell+\ell_0\}$ such that
\begin{equation*}
    \begin{cases}
    \eta^{(0)}=\eta,
    \\\eta^{(i+1)}=(\eta^{(i)})^{x^{(i)}, x^{(i)}+1}, & i=0, \ldots, N(x')-1
    \\ \eta^{(N(x'))}=\eta^{x+y, x+z},
    \end{cases}
\end{equation*}
and such that $c_{x^{(i)}, x^{(i)}+1}(\eta^{(i)})$ is strictly positive for all $i=0, \ldots, N(x')$; an example of this construction is shown in Figure \ref{fig:path}.

%%%%%%%%%%%%%%%%%%%%%%%%%%%%%%%%%%%%%%%%%%%%%%%%%%
\begin{figure}[!ht]
\begin{center}
\begin{tikzpicture}[thick, scale=1][h!]    
    \foreach \y in {0,-1,...,-7}
    {\draw[step=1cm,lightgray,very thin] (-5,\y) grid (5,\y);
    \draw[step=1cm,gray,very thick] (-1,\y) grid (5,\y);}
    \foreach \y in {-9}
    {\draw[step=1cm,lightgray,very thin] (-5,\y) grid (5,\y);
    \draw[step=1cm,gray,very thick] (-1,\y) grid (5,\y);}
    
    \foreach \y in {0,-1,...,-7}
    {\foreach \x in {-5,...,5}{
    \draw[color=lightgray,very thin] (\x,\y-0.2)--(\x,\y+0.2);} 
    \draw[color=gray, very thick] (-1,\y-0.3)--(-1,\y+0.3);
    \draw[color=gray, very thick] (5,\y-0.3)--(5,\y+0.3);}
    \draw (0, -8) node [color=black] {$\vdots$};
    \foreach \y in {-9}
    {\foreach \x in {-5,...,5}{
    \draw[color=lightgray,very thin] (\x,\y-0.2)--(\x,\y+0.2);} 
    \draw[color=gray, very thick] (-1,\y-0.3)--(-1,\y+0.3);
    \draw[color=gray, very thick] (5,\y-0.3)--(5,\y+0.3);}
    
    \draw (-5,0.6) node [color=black] {$\scriptstyle{x+1}$};
    \draw (-4,0.6) node [color=black] {$\scriptstyle{x+y}$};
    \draw (-2,0.6) node [color=black] {$\scriptstyle{x+z}$};
    \draw (-1,0.6) node [color=black] {$\scriptstyle{x+\ell+1}$};
    \draw (5,0.6) node [color=black] {$\scriptstyle{x+\ell+\ell_0}$};
    \draw (3,0.6) node [color=black] {$\scriptstyle{X(\eta)}$};  
     
    %eta
    \draw (5.7,0) node [color=black] {$\eta$};
    \node[shape=circle, minimum size=0.4cm] (-5) at (-5,0) {};
    \node[ball color=black!30!, shape=circle, minimum size=0.4cm] (-4) at (-4,0) {};
    \node[shape=circle, minimum size=0.4cm] (-3) at (-3,0) {};
    \node[shape=circle, minimum size=0.4cm] (-2) at (-2,0) {};
    \node[shape=circle, minimum size=0.4cm] (-1) at (-1,0) {};
    \node[ball color=black!30!, shape=circle, minimum size=0.4cm] (0) at (0,0) {};
    \node[shape=circle, minimum size=0.4cm] (1) at (1,0) {};
    \node[shape=circle, minimum size=0.4cm] (2) at (2,0) {};
    \node[ball color=black!30!, shape=circle, minimum size=0.4cm] (3) at (3,0) {};
    \node[ball color=black!30!, shape=circle, minimum size=0.4cm] (4) at (4,0) {};
    \node[shape=circle, minimum size=0.4cm] (5) at (5,0) {};
    \path [<-] (2) edge[thin, bend left=65] node[above] {} (3);

    %first step
    \node[shape=circle, minimum size=0.4cm] (-5) at (-5,-1) {};
    \node[ball color=black!30!, shape=circle, minimum size=0.4cm] (-4) at (-4,-1) {};
    \node[shape=circle, minimum size=0.4cm] (-3) at (-3,-1) {};
    \node[shape=circle, minimum size=0.4cm] (-2) at (-2,-1) {};
    \node[shape=circle, minimum size=0.4cm] (-1) at (-1,-1) {};
    \node[ball color=black!30!, shape=circle, minimum size=0.4cm] (0) at (0,-1) {};
    \node[shape=circle, minimum size=0.4cm] (1) at (1,-1) {};
    \node[ball color=black!30!, shape=circle, minimum size=0.4cm] (2) at (2,-1) {};
    \node[shape=circle, minimum size=0.4cm] (3) at (3,-1) {};
    \node[ball color=black!30!, shape=circle, minimum size=0.4cm] (4) at (4,-1) {};
    \node[shape=circle, minimum size=0.4cm] (5) at (5,-1) {};
    \path [<-] (1) edge[thin, bend left=65] node[above] {} (2);

    %second step
    \node[shape=circle, minimum size=0.4cm] (-5) at (-5,-2) {};
    \node[ball color=black!30!, shape=circle, minimum size=0.4cm] (-4) at (-4,-2) {};
    \node[shape=circle, minimum size=0.4cm] (-3) at (-3,-2) {};
    \node[shape=circle, minimum size=0.4cm] (-2) at (-2,-2) {};
    \node[shape=circle, minimum size=0.4cm] (-1) at (-1,-2) {};
    \node[ball color=black!30!, shape=circle, minimum size=0.4cm] (0) at (0,-2) {};
    \node[ball color=black!30!, shape=circle, minimum size=0.4cm] (1) at (1,-2) {};
    \node[shape=circle, minimum size=0.4cm] (2) at (2,-2) {};
    \node[shape=circle, minimum size=0.4cm] (3) at (3,-2) {};
    \node[ball color=black!30!, shape=circle, minimum size=0.4cm] (4) at (4,-2) {};
    \node[shape=circle, minimum size=0.4cm] (5) at (5,-2) {};
    \path [<-] (-1) edge[thin, bend left=65] node[above] {} (0);

    %third step
    \node[shape=circle, minimum size=0.4cm] (-5) at (-5,-3) {};
    \node[ball color=black!30!, shape=circle, minimum size=0.4cm] (-4) at (-4,-3) {};
    \node[shape=circle, minimum size=0.4cm] (-3) at (-3,-3) {};
    \node[shape=circle, minimum size=0.4cm] (-2) at (-2,-3) {};
    \node[ball color=black!30!, shape=circle, minimum size=0.4cm] (-1) at (-1,-3) {};
    \node[shape=circle, minimum size=0.4cm] (0) at (0,-3) {};
    \node[ball color=black!30!, shape=circle, minimum size=0.4cm] (1) at (1,-3) {};
    \node[shape=circle, minimum size=0.4cm] (2) at (2,-3) {};
    \node[shape=circle, minimum size=0.4cm] (3) at (3,-3) {};
    \node[ball color=black!30!, shape=circle, minimum size=0.4cm] (4) at (4,-3) {};
    \node[shape=circle, minimum size=0.4cm] (5) at (5,-3) {};
    \path [<-] (0) edge[thin, bend left=65] node[above] {} (1);

    %fourth step
    \node[shape=circle, minimum size=0.4cm] (-5) at (-5,-4) {};
    \node[ball color=black!30!, shape=circle, minimum size=0.4cm] (-4) at (-4,-4) {};
    \node[shape=circle, minimum size=0.4cm] (-3) at (-3,-4) {};
    \node[shape=circle, minimum size=0.4cm] (-2) at (-2,-4) {};
    \node[ball color=black!30!, shape=circle, minimum size=0.4cm] (-1) at (-1,-4) {};
    \node[ball color=black!30!, shape=circle, minimum size=0.4cm] (0) at (0,-4) {};
    \node[shape=circle, minimum size=0.4cm] (1) at (1,-4) {};
    \node[shape=circle, minimum size=0.4cm] (2) at (2,-4) {};
    \node[shape=circle, minimum size=0.4cm] (3) at (3,-4) {};
    \node[ball color=black!30!, shape=circle, minimum size=0.4cm] (4) at (4,-4) {};
    \node[shape=circle, minimum size=0.4cm] (5) at (5,-4) {};
    \path [<-] (-2) edge[thin, bend left=65] node[above] {} (-1);

    %fifth step
    \node[shape=circle, minimum size=0.4cm] (-5) at (-5,-5) {};
    \node[ball color=black!30!, shape=circle, minimum size=0.4cm] (-4) at (-4,-5) {};
    \node[shape=circle, minimum size=0.4cm] (-3) at (-3,-5) {};
    \node[ball color=black!30!, shape=circle, minimum size=0.4cm] (-2) at (-2,-5) {};
    \node[shape=circle, minimum size=0.4cm] (-1) at (-1,-5) {};
    \node[ball color=black!30!, shape=circle, minimum size=0.4cm] (0) at (0,-5) {};
    \node[shape=circle, minimum size=0.4cm] (1) at (1,-5) {};
    \node[shape=circle, minimum size=0.4cm] (2) at (2,-5) {};
    \node[shape=circle, minimum size=0.4cm] (3) at (3,-5) {};
    \node[ball color=black!30!, shape=circle, minimum size=0.4cm] (4) at (4,-5) {};
    \node[shape=circle, minimum size=0.4cm] (5) at (5,-5) {};
    \path [->] (-4) edge[thin, bend left=65] node[above] {} (-3);

    %sixth step
    \node[shape=circle, minimum size=0.4cm] (-5) at (-5,-6) {};
    \node[shape=circle, minimum size=0.4cm] (-4) at (-4,-6) {};
    \node[ball color=black!30!, shape=circle, minimum size=0.4cm] (-3) at (-3,-6) {};
    \node[ball color=black!30!, shape=circle, minimum size=0.4cm] (-2) at (-2,-6) {};
    \node[shape=circle, minimum size=0.4cm] (-1) at (-1,-6) {};
    \node[ball color=black!30!, shape=circle, minimum size=0.4cm] (0) at (0,-6) {};
    \node[shape=circle, minimum size=0.4cm] (1) at (1,-6) {};
    \node[shape=circle, minimum size=0.4cm] (2) at (2,-6) {};
    \node[shape=circle, minimum size=0.4cm] (3) at (3,-6) {};
    \node[ball color=black!30!, shape=circle, minimum size=0.4cm] (4) at (4,-6) {};
    \node[shape=circle, minimum size=0.4cm] (5) at (5,-6) {};
    \path [->] (-2) edge[thin, bend left=65] node[above] {} (-1);

    %seventh step
    \node[shape=circle, minimum size=0.4cm] (-5) at (-5,-7) {};
    \node[shape=circle, minimum size=0.4cm] (-4) at (-4,-7) {};
    \node[ball color=black!30!, shape=circle, minimum size=0.4cm] (-3) at (-3,-7) {};
    \node[shape=circle, minimum size=0.4cm] (-2) at (-2,-7) {};
    \node[ball color=black!30!, shape=circle, minimum size=0.4cm] (-1) at (-1,-7) {};
    \node[ball color=black!30!, shape=circle, minimum size=0.4cm] (0) at (0,-7) {};
    \node[shape=circle, minimum size=0.4cm] (1) at (1,-7) {};
    \node[shape=circle, minimum size=0.4cm] (2) at (2,-7) {};
    \node[shape=circle, minimum size=0.4cm] (3) at (3,-7) {};
    \node[ball color=black!30!, shape=circle, minimum size=0.4cm] (4) at (4,-7) {};
    \node[shape=circle, minimum size=0.4cm] (5) at (5,-7) {};
    \path [->] (-3) edge[thin, bend left=65] node[above] {} (-2);

    %last step
    \draw (6.2,-9) node [color=black] {$\eta^{x+y, x+z}$};
    \node[shape=circle, minimum size=0.4cm] (-5) at (-5,-9) {};
    \node[shape=circle, minimum size=0.4cm] (-4) at (-4,-9) {};
    \node[shape=circle, minimum size=0.4cm] (-3) at (-3,-9) {};
    \node[ball color=black!30!, shape=circle, minimum size=0.4cm] (-2) at (-2,-9) {};
    \node[shape=circle, minimum size=0.4cm] (-1) at (-1,-9) {};
    \node[ball color=black!30!, shape=circle, minimum size=0.4cm] (0) at (0,-9) {};
    \node[shape=circle, minimum size=0.4cm] (1) at (1,-9) {};
    \node[shape=circle, minimum size=0.4cm] (2) at (2,-9) {};
    \node[ball color=black!30!, shape=circle, minimum size=0.4cm] (3) at (3,-9) {};
    \node[ball color=black!30!, shape=circle, minimum size=0.4cm] (4) at (4,-9) {};
    \node[shape=circle, minimum size=0.4cm] (5) at (5,-9) {};
    
\end{tikzpicture}
\caption{Example of a path from a configuration $\eta$ in $\mathscr{G}_{\ell_0}(x+\ell)$ to the configuration $\eta^{x+y, x+z}$ that swaps the occupation state of sites $x+y$ and $x+z$.}\label{fig:path}
\end{center}
\end{figure}
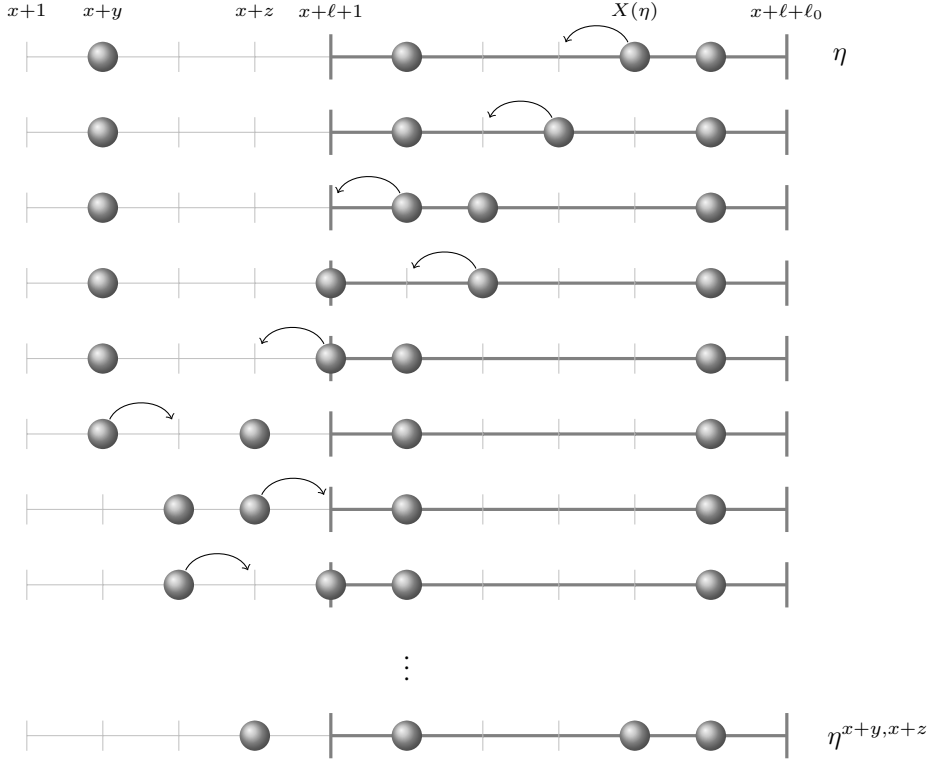
%%%%%%%%%%%%%%%%%%%%%%%%%%%%%%%%%%%%%%%%%%%%%%%%%%

Moreover, this sequence can be chosen in such a way that the number of its elements satisfies $N(x')\le C(\ell+\ell_0)$, and each value in $\{x+1, \ldots, x+\ell+\ell_0\}$ is taken by $x^{(i)}$ for at most six $i$'s; see \cite{bgs17} for details. Hence, we can write
\begin{align*}
    &\int_{\Omega_N}\varphi(\tau_x\eta)[\eta_{x+y}-\eta_{x+z}]\boldsymbol{1}_{\mathscr{G}_{\ell_0}(x+\ell)}(\eta)[f(\eta)-f(\eta^{x+y, x+z})] \de \nu_N^\alpha
    \\&=\sum_{x'\in\Lambda_{x+\ell}^{\ell_0}}\int_{\Omega_N}\varphi(\tau_x\eta)[\eta_{x+y}-\eta_{x+z}]\boldsymbol{1}_{\{X(\eta)=x'\}}\sum_{i=1}^{N(x')}\left[f\big(\eta^{(i-1)}\big)-f\big(\eta^{(i)}\big)\right] \de \nu_N^\alpha.
\end{align*}
Now, by Young's inequality, for any $A_x>0$ the expression above is bounded from above by 
\begin{align*}
    &\frac{A_x}{2}\sum_{x'\in\Lambda_{x+\ell}^{\ell_0}}\int_{\Omega_N}\sum_{i=1}^{N(x')}\varphi(\tau_x\eta)^2\frac{[\eta_{x+y}-\eta_{x+z}]^2}{c_{x^{(i-1)}, x^{(i-1)}+1}^N(\eta^{(i-1)})}\boldsymbol{1}_{\{X(\eta)=x'\}}\de \nu_N^\alpha
    \\&+\frac{1}{2A_x}\sum_{x'\in\Lambda_{x+\ell}^{\ell_0}}\int_{\Omega_N}\sum_{i=1}^{N(x')}c_{x^{(i-1)}, x^{(i-1)}+1}^N(\eta^{(i-1)})\boldsymbol{1}_{\{X(\eta)=x'\}}\hspace{-.3em}\left[f\big(\eta^{(i-1)}\big)-f\big(\eta^{(i)}\big)\right]^2 \de \nu_N^\alpha,
\end{align*}
where, for $N$ sufficiently large, $c_{x^{(i-1)}, x^{(i-1)}+1}^N(\eta^{(i-1)})$ is bounded from below by a positive constant for each $\eta$ and $i$. Hence, the first line can be bounded from above by
\begin{equation*}
    C A_x(\ell+\ell_0)\Var_{\nu_N^\alpha}(\varphi).
\end{equation*}
As for the second line, since any given value is taken by the $x^{(i)}$'s at most six times and since, for $\alpha>1$, $\frac{\nu_N^\alpha(\eta)}{\nu_N^\alpha(\xi)}\lesssim 1$ for any $\eta, \xi\in\Omega_N$, it can be bounded from above by
\begin{equation*}
    \frac{C}{A_x}\int_{\Omega_N}\sum_{k=x+1}^{x+\ell+\ell_0-1}c_{k, k+1}^N(\eta)\left[f\left(\eta^{k, k+1}\right)-f(\eta)\right]^2\de \nu_N^\alpha.
\end{equation*}
By \eqref{eq:dirichlet_KLS}, choosing $A_x=\left|G\left(\frac{x}{N}\right)\right|\frac{\ell+\ell_0}{2CN^2}$ when $G\left(\frac{x}{N}\right)\ne0$ completes the proof.
\end{proof}

%%%%%%%%%%%%%%%%%%%%%%%%%%%%%%%%%%%%%%%%%%%%%%%%%%
%%%Characterisation of the Limit Point%%%%%%%%%%%%
%%%%%%%%%%%%%%%%%%%%%%%%%%%%%%%%%%%%%%%%%%%%%%%%%%
\subsection{Characterisation of the Limit Point}
We are finally ready to study the quadratic term \eqref{eq:quadratic_term} and thus conclude the proof of Theorem \ref{thm:fluctuations}.

\begin{proof}[Proof of Theorem \ref{thm:fluctuations}]
By Lemma \ref{lemma:expectation}, for each $x\in\mathbb{T}_N$ we have that
\begin{align*}
    &\left[\frac{\nu_N^\alpha(\eta^{x, x+1})}{\nu_N^\alpha(\eta)}-1\right](\bar\eta_x-\bar\eta_{x+1})
    \\&=\left[e^{-N^{-\alpha}(\bar\eta_x-\bar\eta_{x+1})(\bar\eta_{x+2}-\bar\eta_{x-1})}-1\right](\bar\eta_x-\bar\eta_{x+1})
    \\&=-N^{-\alpha}(\bar\eta_x-\bar\eta_{x+1})^2(\bar\eta_{x+2}-\bar\eta_{x-1})+O(N^{-2\alpha})
    \\&=\frac{1}{4}N^{-\alpha}(\bar\eta_{x-1}-\bar\eta_{x+2})+2N^{-\alpha}(\bar\eta_{x}\bar\eta_{x+1}\bar\eta_{x+2}-\bar\eta_{x-1}\bar\eta_{x}\bar\eta_{x+1})+O(N^{-2\alpha}),
\end{align*}
and thus the stationary measure $\nu_N^\alpha$ satisfies \textbf{\textsc{Assumption 1a}}. Moreover, by Lemma \ref{lemma:correlations}, it also satisfies \textbf{\textsc{Assumption 2}} with $\beta=\alpha$. 

Now, note that if $\varphi$ is uniformly bounded, then the first term on the RHS of \eqref{eq:restriction_bound} vanishes \textit{exponentially} in $N$ for any choice of $\ell_0$ which is polynomial in $N$. But then, by Lemmas \ref{lemma:restriction} and \ref{lemma:path}, it is not hard to check that, choosing for instance $\ell_0=N^{\frac{1}{3}}$, up to a term vanishing exponentially in $N$, we can apply the bound \eqref{eq:bg_boundA} from Theorem \ref{thm:gen_bg}, and in particular choosing $L=\eps N$ we get
\begin{align*}
    &\expected_{\nu_N^\alpha}\left[\left(\int_0^t\sum_{x\in \Lambda_N}\nabla_N\phi\Bigl(\frac{x-\boldsymbol{v}_N s}{N}\Bigr)\left\{\bar\eta_{sN^2}(x)\bar\eta_{sN^2}(x+j)-\cev{\eta}_{sN^2}^{\eps N}(x)\vec{\eta}_{sN^2}^{\eps N}(x)\right\}\de s\right)^2\right]
    \\&\lesssim \eps t\|\nabla_N\phi\|_{2, N}^2 \bigg\{1+\frac{1}{N^{\alpha-1}}+\frac{t}{N^{2\alpha-2}}+\frac{t}{N^{3\alpha-4}}+\frac{t}{N^{4\alpha-3}}\bigg\}
\end{align*}
for $j\in\{1, 2\}$, which vanishes as $N\to\infty$ and $\eps\to0$ for any $\alpha\ge\frac{4}{3}$. Hence, up to a term vanishing in $L^2(\prob_{\nu_N^\alpha})$ as $N\to\infty$ and $\eps\to0$, we can rewrite the quadratic term \ref{eq:quadratic_term} as
\begin{equation}\label{eq:quadratic_term_2}
    (2b-6b\bar\rho_N)N^{\frac{1}{2}-\gamma}\int_0^t\sum_{x\in\mathbb{T}_N}\cev{\eta}_{sN^2}^{\eps N}(x)\vec{\eta}_{sN^2}^{\eps N}(x)\nabla_N\phi\Bigl(\frac{x-\boldsymbol{v}_N s}{N}\Bigr)\de s.
\end{equation}
Now:
\begin{enumerate}[i)]
    \item if $\gamma>\frac{1}{2}$, this term again vanishes as $N\to\infty$: indeed, by the Cauchy-Schwarz inequality and Lemma \ref{lemma:correlations}, it is not hard to check that
    \begin{align}
        &\expected_{\nu_N^\alpha}\left[\left((2b-6b\bar\rho_N)N^{\frac{1}{2}-\gamma}\int_0^t\sum_{x\in\mathbb{T}_N}\cev{\eta}_{sN^2}^{\eps N}(x)\vec{\eta}_{sN^2}^{\eps N}(x)\nabla_N\phi\Bigl(\frac{x-\boldsymbol{v}_N s}{N}\Bigr)\de s\right)^2\right]\nonumber
        \\&\lesssim t^2 N^{1-2\gamma}\|\nabla\phi\|_\infty^2 E_{\nu_N^\alpha}\left[\left(\sum_{x\in\mathbb{T}_N}\cev{\eta}^{\eps N}(x)\vec{\eta}^{\eps N}(x)\right)^2\right]\nonumber
        \\\begin{split}& \lesssim t^2  N^{1-2\gamma}\|\nabla\phi\|_\infty^2\eps N\sum_{x\in\mathbb{T}_N} E_{\nu_N^\alpha}\left[\left(\cev{\eta}^{\eps N}(x)\vec{\eta}^{\eps N}(x)\right)^2\right]
        \\&\phantom{\lesssim}+t^2  N^{1-2\gamma}\|\nabla\phi\|_\infty^2\sum_{|x-y|>2\eps N} E_{\nu_N^\alpha}\left[\cev{\eta}^{\eps N}(x)\vec{\eta}^{\eps N}(x)\cev{\eta}^{\eps N}(y)\vec{\eta}^{\eps N}(y)\right].
        \end{split}\label{eq:quadratic_twosum}
    \end{align}
    
    Note now that, for each $x\in\mathbb{T}_N$,
    \begin{align}
        &E_{\nu_N^\alpha}\left[\left(\cev{\eta}^{\eps N}(x)\vec{\eta}^{\eps N}(x)\right)^2\right]\nonumber
        \\&=\frac{1}{\eps^4N^4} E_{\nu_N^\alpha}\left[(\bar\eta_{x-\eps N}+\ldots + \bar\eta_{x-1})^2(\bar\eta_{x+1}+\ldots+\bar\eta_{x+\eps N})^2\right]\label{eq:expectation_terms}
        \\&\lesssim \frac{1}{\eps^2N^2}+\frac{N^{-\alpha}}{\eps N}+N^{-2\alpha}.\nonumber
    \end{align}
    Here, we used the fact that the number of correlation terms coming from the expectation in \eqref{eq:expectation_terms} which are at minimum distance one -- and thus have correlation of order at most $N^{-\alpha}$ -- is $\eps^3N^3$; for all the other terms, the correlation improves to $N^{-2\alpha}$. Similarly, we see that, for each $|x-y|>2\eps N$,
    \begin{equation*}
        E_{\nu_N^\alpha}\left[\cev{\eta}^{\eps N}(x)\vec{\eta}^{\eps N}(x)\cev{\eta}^{\eps N}(y)\vec{\eta}^{\eps N}(y)\right]\lesssim \frac{N^{-\alpha}}{\eps N}+N^{-2\alpha}.
    \end{equation*}
    This yields
    \begin{equation*}
        \eqref{eq:quadratic_twosum}\lesssim t^2  \|\nabla\phi\|_\infty^2\left\{\frac{N^{1-2\gamma}}{\eps}+\left(1+\frac{1}{\eps}\right)N^{2-2\gamma-\alpha}+(\eps+1) N^{3-2\gamma-2\alpha}\right\}. 
    \end{equation*}
    Using the computations made in Section \ref{sec:dynkin} and Proposition \ref{prop:fixed_time}, a standard argument shows that any limit point $\{\mathscr{Y}_t, t\in[0, T]\}$ of $\{\mathscr{Y}_t^N, t\in[0, T]\}_N$ is concentrated on the unique stationary solution of \eqref{eq:limit_OU};

    \item if, instead, $\gamma=\frac{1}{2}$, then \eqref{eq:quadratic_term_2} can be rewritten as
    \begin{equation*}
        (2b-6b\bar\rho_N)\int_0^t\frac{1}{N}\sum_{x\in\mathbb{T}_N}\nabla_N\phi\Bigl(\frac{x-\boldsymbol{v}_N s}{N}\Bigr)\mathscr{Y}_s^N\left(\iota_\eps\Bigl(\frac{x}{N}\Bigr)\right)^2\de s,
    \end{equation*}
    where $\iota_\eps(u)(\cdot)$ is a continuous approximation of $\frac{1}{\eps}\boldsymbol{1}_{(u, u+\eps]}(\cdot)$. Together with the computations made in Section \ref{sec:dynkin} and Proposition \ref{prop:fixed_time}, this shows that any limit point $\{\mathscr{Y}_t, t\in[0, T]\}$ of $\{\mathscr{Y}_t^N, t\in[0, T]\}_N$ satisfies items i) and iii) of Definition \ref{def:energy_solutions}. Item ii) follows again from Theorem \ref{thm:gen_bg}; see \cite{gj14} for a full, analogue argument. Finally, item iv) follows from repeating the computations in Section \ref{sec:dynkin} for the reversed dynamics, namely the process generated by the adjoint of $\mathcal{L}^N$ with respect to $\nu_N^\alpha$.

\end{enumerate}
This completes the proof.
\end{proof}

\appendix

%%%%%%%%%%%%%%%%%%%%%%%%%%%%%%%%%%%%%%%%%%%%%%%%%%
%%%ON THE INVARIANT MEASURE OF THE KLS MODEL%%%%%%
%%%%%%%%%%%%%%%%%%%%%%%%%%%%%%%%%%%%%%%%%%%%%%%%%%
\section{On the Invariant Measure of the KLS Model}\label{sec:appendix}
This appendix is dedicated to the study of the stationary measure of the KLS model. For the sake of completeness, we first give a simple proof of the invariance of the measure  $\nu_N^\alpha$ for the dynamics of the KLS model for our choice of parameters. We then show how to derive its desired properties (Lemmas \ref{lemma:expectation} and \ref{lemma:correlations}) via the \textit{transfer matrix} approach, and we conclude by showing Proposition \ref{prop:fixed_time} via a CLT for $\alpha$-mixing random variables.

%%%%%%%%%%%%%%%%%%%%%%%%%%%%%%%%%%%%%%%%%%%%%%%%%%
%%%Stationarity of the Measure%%%%%%%%%%%%%%%%%%%%
%%%%%%%%%%%%%%%%%%%%%%%%%%%%%%%%%%%%%%%%%%%%%%%%%%
\subsection{Stationarity of the Measure}
\begin{proposition} The measure $\nu_N^\alpha$ defined in \eqref{eq:nuN_alpha} is invariant for the dynamics of the KLS model with parameters \eqref{eq:params}.
\end{proposition}

\begin{proof} We will actually show that the measure \eqref{eq:nuN} satisfies the \textit{global balance} condition, namely
\begin{equation}\label{eq:global_balance}
    \sum_{x\in\mathbb{T}_N}\left[\nu_N(\eta^{x, x+1})c_{x, x+1}^N(\eta^{x, x+1})-\nu_N(\eta)c_{x, x+1}^N(\eta)\right]=0
\end{equation}
for each $\eta\in\Omega_N$, where the rates $c_{x, x+1}^N$ are given by \eqref{eq:rates} with $\kappa_N=-1$. In particular, this implies that $\int_{\Omega_N} \mathcal{L}^Nf\de \nu_N^\alpha=0$ for each cylinder function $f:\Omega_N\to\mathbb{R}$, and thus the measure $\nu_N^\alpha$ is invariant for the dynamics with parameters \eqref{eq:params}.

Calling $1+\eps_N=:a$ and $1-\eps_N=:b$, in order to show \eqref{eq:global_balance}, it suffices to prove that, for each $\eta\in\Omega_N$,
\begin{equation*}
    \sum_{x\in\mathbb{T}_N} \eta_{x+1}(1-\eta_x)[a\eta_{x-1}+b\eta_{x+2}]\frac{\nu_N(\eta^{x, x+1})}{\nu_N(\eta)}=\sum_{x\in\mathbb{T}_N}\eta_x(1-\eta_{x+1})[a\eta_{x-1}+b\eta_{x+2}]
\end{equation*}
and
\begin{equation*}
    \sum_{x\in\mathbb{T}_N} \eta_x(1-\eta_{x+1})[b\eta_{x-1}+a\eta_{x+2}]\frac{\nu_N(\eta^{x, x+1})}{\nu_N(\eta)}=\sum_{x\in\mathbb{T}_N}\eta_{x+1}(1-\eta_x)[b\eta_{x-1}+a\eta_{x+2}].
\end{equation*}
We will only show {the first identity, as the proof of the second one} is completely analogous. 

Note that 
\begin{equation*}
    \frac{\nu_N(\eta^{x, x+1})}{\nu_N(\eta)}=\Bigl(\frac{b}{a}\Bigr)^{(\eta_x-\eta_{x+1})(\eta_{x+2}-\eta_{x-1})}.
\end{equation*}
Now, on the event $\{\eta_{x+1}=1, \eta_x=0\}$, the expression above reads
\begin{equation*}
    \Bigl(\frac{b}{a}\Bigr)^{\eta_{x-1}-\eta_{x+2}}=\frac{b}{a}\eta_{x-1}(1-\eta_{x+2})+\frac{a}{b}\eta_{x+2}(1-\eta_{x-1})+1-\eta_{x-1}-\eta_{x+2}+2\eta_{x-1}\eta_{x+2},
\end{equation*}
and hence, on that event, we have that
\begin{align*}
    [a\eta_{x-1}+b\eta_{x+2}]\frac{\nu_N(\eta^{x, x+1})}{\nu_N(\eta)}&=b\eta_{x-1}(1-\eta_{x+2})+(a+b)\eta_{x-1}\eta_{x+2}+a\eta_{x+2}(1-\eta_{x-1})
    \\&=a\eta_{x+2}+b\eta_{x-1}.
\end{align*}
But then, it suffices to prove that
\begin{equation*}
    \sum_{x\in\mathbb{T}_N} \eta_{x+1}(1-\eta_x)[b\eta_{x-1}+a\eta_{x+2}]=\sum_{x\in\mathbb{T}_N}\eta_x(1-\eta_{x+1})[a\eta_{x-1}+b\eta_{x+2}],
\end{equation*}
which follows immediately by performing a shift in the summation index.
\end{proof}

%%%%%%%%%%%%%%%%%%%%%%%%%%%%%%%%%%%%%%%%%%%%%%%%%%
%%%Expectations via Transfer Matrix%%%%%%%%%%%%%%%
%%%%%%%%%%%%%%%%%%%%%%%%%%%%%%%%%%%%%%%%%%%%%%%%%%
\subsection{Expectations via the Transfer Matrix}
Define the \textit{transfer matrix} of the measure $\nu_N^\alpha$ as
\begin{equation}\label{def:transfer_matrix}
    T:=\begin{bmatrix}
    1 & 1 \\ 1 & e^{-N^{-\alpha}}
    \end{bmatrix}.
\end{equation}
This matrix encodes the particle interactions by evaluating the contribution of the term $\eta_x\eta_{x+1}$ in $e^{-N^{-\alpha}H(\eta)}$ in the four distinct cases $(\eta_x, \eta_{x+1})=\{(0, 0), (0, 1), (1, 0), (1, 1)\}$, and the idea behind the transfer matrix approach is to use the matrix $T$ to derive simple expressions of expected values with respect to $\nu_N^\alpha$; see \cite{bax82} for a more general, detailed exposition. Since $T$ is symmetric, it can be diagonalised, and in particular we can write $T=Q\Lambda Q^{-1}$ with 
\begin{equation}\label{eq:matrices}
    Q=\begin{bmatrix} 1 & 1 \\ \lambda_+-1 & \lambda_--1
        
    \end{bmatrix}, \ \ \ \ \ \ \Lambda= \begin{bmatrix}
        \lambda_+ & 0 \\ 0 & \lambda_-
    \end{bmatrix},
\end{equation}
where
\begin{equation}\label{eq:lambdas}
    \lambda_\pm=\frac{(1+e^{-N^{-\alpha}})\pm\sqrt{(1-e^{-N^{-\alpha}})^2+4}}{2}.
\end{equation}

We will now make use of the matrix $T$ to show Lemmas \ref{lemma:partition}, \ref{lemma:expectation} and \ref{lemma:correlations}.

\begin{proof}[Proof of Lemma \ref{lemma:partition}]
    Note that the partition function $Z_N^\alpha$ of $\nu_N^\alpha$ can be written as
\begin{equation}\label{eq:partition_expression}
    Z_N^\alpha=\text{Tr}(T^N)=\lambda_+^N+\lambda_-^N,
\end{equation}
where $\lambda_+, \lambda_-$ are the eigenvalues of $T$, given in \eqref{eq:lambdas}. Since $|\lambda_-|<1$, we have that
\begin{align*}
    \mylim_{N\to\infty}\mylog \frac{Z_N^\alpha}{2^N}&=\mylim_{N\to\infty}\mylog\frac{\left((1+e^{-N^{-\alpha}})+\sqrt{(1-e^{-N^{-\alpha}})^2+4}\right)^N}{4^N}
    \\&=\mylim_{N\to\infty} N \mylog\left(1+\frac{(1+e^{-N^{-\alpha}})+\sqrt{(1-e^{-N^{-\alpha}})^2+4}-4}{4}\right)
    \\&=\mylim_{N\to\infty} N f_\alpha(N),
\end{align*}
where $f_\alpha(N)\lesssim N^{-\alpha}$. Since $\alpha>1$, this completes the proof.
\end{proof}

\begin{proof}[Proof of Lemma \ref{lemma:expectation}]
Note that, for each $x\in\mathbb{T}_N$, we can write
\begin{equation}\label{eq:moment_expression}
    E_{\nu_\alpha^N}[\eta_x]=\frac{(T^N)_{1, 1}}{Z_N^\alpha}.
\end{equation}
Using the matrix $Q$ in \eqref{eq:matrices} and the expression \eqref{eq:partition_expression} for the partition function, since $\lambda_+$ is strictly greater than $\lambda_-$ in absolute value, this implies that
\begin{align*}
    \left|E_{\nu_\alpha^N}[\eta_x]-\frac{1}{2}\right|&=\left|\frac{\lambda_+^N(\lambda_+-1)-\lambda_-^N(\lambda_--1)}{(\lambda_+-\lambda_-)(\lambda_+^N+\lambda_-^N)}-\frac{1}{2}\right|
    \\&\lesssim\left|\frac{\lambda_+-1}{\lambda_+-\lambda_-}-\frac{1}{2}\right|
    \\&=\left|\frac{e^{-N^{-\alpha}}-1+\sqrt{(1-e^{-N^{-\alpha}})^2+4}}{2\sqrt{(1-e^{-N^{-\alpha}})^2+4}}-\frac{1}{2}\right|
    \\&\lesssim N^{-\alpha},
\end{align*}
as claimed.
\end{proof}

\begin{proof}[Proof of Lemma \ref{lemma:correlations}]
Let $x_1<\ldots<x_m\in\mathbb{T}_N$ and $d_1, \ldots, d_m$ be as in the statement: the key observation is that we can write
\begin{equation*}
    E_{\nu_N^\alpha}\left[\prod_{i=1}^m\eta_{x_m}\right]=\frac{\prod_{i=1}^m (T^{d_i})_{1, 1}}{Z_N^\alpha}.
\end{equation*}

We only show the result for $m=2$, as the argument runs the same for any value of $m$ (though the computations are more tedious). By the expression above, using the matrix $Q$ in \eqref{eq:matrices} and the expression \eqref{eq:partition_expression}, for each $x_1<x_2$ in $\in\mathbb{T}_N$ we have that
\begin{align*}
    E_{\nu_\alpha^N}[\eta_{x_1}\eta_{x_2}]&=\frac{(T^{d_1})_{1,1}(T^{d_2})_{1, 1}}{\text{Tr}(T^N)}
    \\&=\frac{\left[\lambda_+^{d_1}(\lambda_+-1)-\lambda_-^{d_1}(\lambda_--1)\right]\left[\lambda_+^{d_2}(\lambda_+-1)-\lambda_-^{d_2}(\lambda_--1)\right]}{(\lambda_+-\lambda_-)^2(\lambda_+^N+\lambda_-^N)}.
\end{align*}
Calling $\lambda_+-1=:a$ and $\lambda_--1=:b$, using \eqref{eq:moment_expression} and recalling that $d_1+d_2=N$, an easy computation shows that
\begin{align*}
    E_{\nu_\alpha^N}&[\eta_{x_1}\eta_{x_2}]-E_{\nu_\alpha^N}[\eta_{x_1}]^2
    \\& =\frac{-ab(\lambda_+^{N+d_1}\lambda_-^{d_2}+\lambda_+^{N+d_2}\lambda_-^{d_1}+\lambda_+^{d_1}\lambda_-^{N+d_2}+\lambda_+^{d_2}\lambda_-^{N+d_1}+2\lambda_+^N\lambda_-^N)}{(a-b)^2(\lambda_+^N+\lambda_-^N)^2}.
\end{align*}
But now, since that $\lambda_+>1$ and $\lambda_-=-\frac{1}{2}N^{-\alpha}+O(N^{-2\alpha})$, this yields
\begin{align*}
    \left|E_{\nu_\alpha^N}[\eta_{x_1}\eta_{x_2}]-E_{\nu_\alpha^N}[\eta_{x_1}]^2\right|&\lesssim \lambda_+^{d_1-N}\lambda_-^{d_2}+\lambda_+^{d_2-N}\lambda_-^{d_1}
    \\&=\left(\frac{\lambda_-}{\lambda_+}\right)^{d_1}+\left(\frac{\lambda_-}{\lambda_+}\right)^{d_2}
    \\&\lesssim N^{-\alpha\mymin\{d_1, d_2\}},
\end{align*}
as claimed. \end{proof}

%%%%%%%%%%%%%%%%%%%%%%%%%%%%%%%%%%%%%%%%%%%%%%%%%%
%%%An Application of a CLT for α-Mixing Sequences%
%%%%%%%%%%%%%%%%%%%%%%%%%%%%%%%%%%%%%%%%%%%%%%%%%%
\subsection[An Application of a CLT for $\alpha$-Mixing Sequences]{An Application of a CLT for $\boldsymbol{\alpha}$-Mixing Sequences}
We conclude the section by showing Proposition \ref{prop:fixed_time}. Note that the random variables $\{\bar\eta_x\}_x$ are identically distributed but \textit{not} independent, so one cannot apply the classical CLT to show convergence of the density fluctuation field at a fixed time to a Gaussian random variable. However, as highlighted by Lemma \ref{lemma:correlations}, this sequence is very weakly dependent, in the sense that its correlation functions quickly decay in space: as we will show, this allows us to apply a generalisation of the CLT to \textit{$\alpha$-mixing} sequences of random variables; see \cite{bra07} for an exhaustive exposition of CLT's under mixing conditions.

Given a probability space $(\Omega, \mathcal{F}, \prob)$, we define the \textit{$\alpha$-mixing coefficient} of two sub-sigma algebras $\mathcal{A}, \mathcal{B}\subset\mathcal{F}$ as 
\begin{align}\label{def:coefficients}
    &\alpha(\mathcal{A}, \mathcal{B}):=\mysup_{A\in\mathcal{A}, B\in\mathcal{B}}\left |\prob(A\cap B)-\prob(A)\prob(B)\right|.
\end{align}
On the torus $\mathbb{T}_n:=\{1, \ldots, n\}$, let the distance between to subsets $I, J\subset\mathbb{T}_n$ be
\begin{equation*}
    d(I, J):=\mymin_{i\in I, j\in J}\{|i-j|\wedge (n-|i-j|)\}.
\end{equation*}
Given a triangular array $\{X_{n, k}\}_{n\in\mathbb{N}, k\in\mathbb{T}_n}$ of random variables, 
we define its \textit{$\alpha$-mixing coefficients} as 
\begin{align*}
    &\alpha_n(r):=\mysup_{\substack{I, J\subset \mathbb{T}_n \\ d(I, J)\ge r}}\alpha\big(\sigma\{\eta_i: i\in I\}, \sigma\{\eta_j: j\in J\}\big)
\end{align*}
for each integer $1\le r\le \lfloor\frac{n}{2}\rfloor$, with the convention $\alpha_n(0):=\frac{1}{4}$. {The result that follows can be found in \cite[Corollary 1]{rio97}.} 

\begin{theorem}\label{thm:alpha_clt} Let $\{X_{n, k}\}_{n\in\mathbb{N}, k\in\mathbb{T}_n}$ be a triangular array of centred random variables with finite variance, and call $Q_{n, k}$ the càdlàg inverse of $t\mapsto \prob\{|X_{n, k}|>t\}$, {namely $Q_{n, k}(u):=\mysup\{t\in[0, \infty): \prob\{|X_{n, k}|>t\}>u\}$}. Let $\{\alpha_n(r)\}_r$ denote the sequence of $\alpha$-mixing coefficients of $\{X_{n, k}\}_k$, and let $\alpha_n^{-1}$ be the càdlàg inverse of $t\mapsto \alpha_n(\lfloor t\rfloor)$. Moreover, calling $\sigma_n^2:=\Var\left(\sum_{k=1}^nX_{n, k}\right)$, assume that
\begin{equation}\label{condition:variance_ratio}
    \mylimsup_{n\to\infty}\mymax_{k, m=1, \ldots, n}\frac{\Var\big(\sum_{j=k}^mX_{n, j}\big)}{\sigma_n^2}<\infty
\end{equation}
and
\begin{equation}\label{condition:lindeberg}
    \mylim_{n\to\infty} \sigma_n^{-3}\sum_{k=1}^n\int_0^1\alpha_n^{-1}\Bigl(\frac{u}{2}\Bigr)Q_{n, k}^2(u)\mymin\left\{\alpha_n^{-1}\Bigl(\frac{u}{2}\Bigr)Q_{n, k}(u), \sigma_n\right\}\de u=0.
\end{equation}
Then
\begin{equation*}
    \frac{\sum_{k=1}^n X_{n, k}}{\sigma_n}\to\mathcal{N}(0, 1)
\end{equation*}
in distribution as $n\to\infty$.
\end{theorem}

\begin{remark}
Condition \eqref{condition:lindeberg} plays the role of a Lindeberg-type control adapted to the mixing framework. In the independent case, the Berry-Esseen theorem requires that the normalised third moments vanish, namely
\begin{equation*}
    \sigma_n^{-3}\sum_{k=1}^n \mathbb{E}|X_{n, k}|^3 \to0
\end{equation*}
as $n\to\infty$, to ensure that no single summand, or small group of summands, dominates the fluctuations; see, for example, \cite[Theorem 1.1]{bbg96}. Under dependence, raw moments are no longer sufficient because large values can be reinforced by correlations. The integral involving the quantile functions $Q_{n,k}$ and the mixing penalty $\alpha_n^{-1}$ is a somewhat natural generalisation: it measures the contribution of the tails of each $X_{n, k}$ and amplifies them according to the strength of dependence. The presence of $\mymin\{\alpha_n^{-1}(\frac{u}{2})Q_{n, k}(u), \sigma_n\}$ compares this aggregated size with the global variance $\sigma_n$, ensuring that only ``genuinely large" deviations (relative to the Gaussian fluctuation scale) are penalised. Intuitively, then, \eqref{condition:lindeberg} rules out the possibility that a few  heavy-tailed or strongly-dependent terms could upset the Gaussian limit, and guarantees that the asymptotic fluctuations arise from the collective effect of many small, weakly correlated contributions.
\end{remark}

\begin{proof}[Proof of Proposition \ref{prop:fixed_time}] We will show that, for each continuous $g:\mathbb{T}\to\mathbb{R}$, the sequence $\{Y_{N, x}\}_{N\in\mathbb{N}, x\in\mathbb{T}_N}$ defined via
\begin{equation*}
    Y_{N, x}:=\frac{1}{\sqrt{N}}g\Bigl(\frac{x}{N}\Bigr)\bar\eta_x^N
\end{equation*}
satisfies the hypotheses of Theorem \ref{thm:alpha_clt}. First note that, by Lemmas \ref{lemma:expectation} and \ref{lemma:correlations},
\begin{align*}
    \sigma_N^2&:=\Var_{\nu_N^\alpha}\left(\sum_{x\in\mathbb{T}_N} Y_{N, x}\right)
    \\&=\frac{1}{N}\sum_{x\in\mathbb{T}_N} g\Bigl(\frac{x}{N}\Bigr)^2E_{\nu_N^\alpha}[\bar\eta_x^2]+\frac{1}{N}\sum_{x\ne y}g\Bigl(\frac{x}{N}\Bigr)g\Bigl(\frac{y}{N}\Bigr)E_{\nu_N^\alpha}[\bar\eta_x\bar\eta_y]
    \\&=\frac{1}{N}\sum_{x\in\mathbb{T}_N} g\Bigl(\frac{x}{N}\Bigr)^2\left(\frac{1}{4}+O(N^{-\alpha})\right)+O\left(\|g\|_\infty^2 N^{-\alpha}\right)  
    \\&\to \frac{1}{4}\|g\|_2^2.
\end{align*}
{Similarly, for each $k, m\in\mathbb{T}_N$, we see that
\begin{equation*}
    \Var_{\nu_N^\alpha}\left(\sum_{x=k}^m Y_{N, x}\right)=\frac{1}{N}\sum_{x=k}^m g\Bigl(\frac{x}{N}\Bigr)^2\left(\frac{1}{4}+O(N^{-\alpha})\right)+O\left(\|g\|_\infty^2 N^{-\alpha}\right)  \lesssim 1,
\end{equation*}}
and thus \eqref{condition:variance_ratio} is automatically satisfied. 

We now claim that the $\alpha$-mixing coefficients of $\{\eta_x\}_{x\in\mathbb{T}_N}$ satisfy
\begin{equation}\label{eq:alpha_bound}
    \alpha_N(r)\lesssim N^{-\alpha r}, \ \ \ \ 0\le r \le \lfloor\tfrac{N}{2}\rfloor.
\end{equation}
Consider the transfer matrix $T$ defined in \eqref{def:transfer_matrix}, and let $v:=[1, \lambda_+-1]$ be an eigenvector -- corresponding to the first column of $Q$ in \eqref{eq:matrices} -- of the eigenvalue $\lambda_+$. We define the $2\times 2$ matrix $P$ by $P_{i, j}:=\frac{v_j}{\lambda_+ v_i}T_{i, j}$: it is then immediate to check that $P$ is a transition matrix. Note that we can view $\{\eta_x\}_{x\in\mathbb{T}_N}$ as a two-state Markov chain with transition matrix $P$ started at the uniform distribution $\mu=[\frac{1}{2}, \frac{1}{2}]$ and conditioned to return to its starting point after $N$ steps: indeed, denoting by $\{\xi_x\}_x$ a stationary Markov chain with transition probabilities $P$, we have that
\begin{align}
    \frac{\prob_\mu\{({\xi_0}, \ldots, \xi_N)=\eta\}}{\prob_\mu\{\xi_0=\xi_N\}}&= \frac{\mu(\eta_0)\prod_{x=0}^{N-1}P_{\eta_x, \eta_{x+1}}}{\sum_{j\in\{0, 1\}}\mu(j)(P^N)_{j, j}}\nonumber
    \\&= \frac{\prod_{x=0}^{N-1}e^{-N^{-\alpha}\eta_x\eta_{x+1}}}{\lambda_+^N(1+\rho_N^N)},\label{eq:markov_prob}
\end{align}
where $\rho_N$ is the smallest eigenvalue of $P$. It is immediate to check that $\rho_N=\frac{\lambda_-}{\lambda_+}$, so that, using \eqref{eq:partition_expression}, we conclude that $\eqref{eq:markov_prob}=\nu_N^\alpha(\eta)$. Now, since $\{\eta_x\}_x$ is an ergodic two-state Markov chain with smallest eigenvalue $\rho_N$, we have that, for each integer $0\le r\le \lfloor\frac{N}{2}\rfloor$, 
$\alpha_N(r)\lesssim |\rho_N|^r\lesssim N^{-\alpha r}$,
where we used $\lambda_-=-\frac{1}{2}N^{-\alpha}+O(N^{-2\alpha})$. 

We are now ready to study the expression in \eqref{condition:lindeberg}. By Markov's inequality, it is easy to see that, for all $x\in\mathbb{T}_N$ and any $\ell>0$,
\begin{equation*}
    Q_{N, x}(u)\le \frac{E_{\nu_N^\alpha}[|Y_{N, k}|^\ell]^{\frac{1}{\ell}}}{u^\frac{1}{\ell}}\lesssim \frac{\|g\|_\infty}{\sqrt{N}}u^{-\frac{1}{\ell}}.
\end{equation*}
Moreover, for each integer $1\le r\le \lfloor\frac{N}{2}\rfloor$, we have that $\alpha_N^{-1}(\frac{u}{2})=r$ if and only if $\frac{u}{2}\in[\alpha_N(r), \alpha_N(r-1))$, and hence, using that $\sigma_N\lesssim 1$,
\begin{align*}
    &\sigma_N^{-3}\sum_{x=1}^N\int_0^1\alpha_N^{-1}\Bigl(\frac{u}{2}\Bigr)Q_{N, x}^2(u)\mymin\left\{\alpha_N^{-1}\Bigl(\frac{u}{2}\Bigr)Q_{N, x}(u), \sigma_N\right\}\de u
    \\&\lesssim \frac{1}{N^\frac{3}{2}}\sum_{x=1}^N\sum_{r=1}^{\lfloor \frac{N}{2}\rfloor} r^2 \int_{2\alpha_N(r)}^{2\alpha_N(r-1)}u^{-\frac{2}{\ell}} \de u
    \\& \lesssim \frac{1}{\sqrt{N}}\sum_{r=1}^{\lfloor \frac{N}{2}\rfloor} r^2 \left[\alpha_N(r-1)^{1-\frac{2}{\ell}}-\alpha_N(r)^{1-\frac{2}{\ell}}\right]
    \\& \lesssim \frac{1}{\sqrt{N}} \sum_{r=1}^{\lfloor \frac{N}{2}\rfloor} r^2 \alpha_N(r-1)^{1-\frac{2}{\ell}}.
\end{align*}
But then, by \eqref{eq:alpha_bound}, the expression in \eqref{condition:lindeberg} is bounded from above by a constant times
\begin{equation*}
    \frac{1}{\sqrt{N}}\sum_{r=1}^{\lfloor \frac{N}{2}\rfloor } r^2 N^{-\alpha (r-1)(1-\frac{2}{\ell})},
\end{equation*}
which is $\lesssim\frac{1}{\sqrt{N}}$, and thus vanishes, for any $\ell>2$; this implies that \eqref{condition:lindeberg} is satisfied. 

Hence, by the convergence of $\sigma_N^2$, we finally get that
\begin{equation*}
    \sum_{x=1}^N Y_{N, x}\to \mathcal{N}\left(0, \frac{1}{4}\|g\|_2^2\right),
\end{equation*}
in distribution as $N\to\infty$, as claimed.
\end{proof}

%%%%%%%%%%%%%%%%%%%%%%%%%%%%%%%%%%%%%%%%%%%%%%%%%%
%%%REFERENCES%%%%%%%%%%%%%%%%%%%%%%%%%%%%%%%%%%%%%
%%%%%%%%%%%%%%%%%%%%%%%%%%%%%%%%%%%%%%%%%%%%%%%%%%
\bibliographystyle{Martin} 
\bibliography{references}

\end{document}